\documentclass[12pt]{amsart}
\usepackage{amssymb}
\usepackage{amsmath}
\usepackage{pictex}
\newcommand\ab{\mathsf{a}}
\newcommand\Aff{\text{\sf Aff}}
\newcommand\al{\alpha}
\newcommand\AND{\quad\text{and}\quad}
\newcommand\bd{\partial}

\newcommand\DL{\mathsf{DL}}
\newcommand\dist{\mathsf{d}}

\newcommand\Ex{\mathsf{E}}
\newcommand\gab{\boldsymbol\gamma}
\newcommand\gb{\mathbf{g}}
\newcommand\gf{\mathfrak{g}}
\newcommand\hb{\mathbf h}
\newcommand\Hb{{\mathbb H}}
\newcommand\Hc{{\mathcal H}}
\newcommand\Hf{{\mathfrak H}}
\newcommand\hor{\wt\pi}%\mathfrak{h}}
\newcommand\HT{\operatorname{\sf HT}}

\newcommand\im{\mathfrak{i}\,}
\newcommand\kb{\mathbf k}
\newcommand\Lap{\mathfrak L}
\newcommand\la{\lambda}
\newcommand\Mart{\mathcal{M}}
\newcommand\ms{\mathbf m}
\newcommand\of{\mathfrak{o}}
\newcommand\Prob{\mathsf{Pr}}
\newcommand\pb{\mathbf p}
\newcommand\pp{\mathsf{p}}
\newcommand\qq{\mathsf{q}}
\newcommand\R{\mathbb{R}}
\newcommand\Sol{\text{\sf Sol}}
\newcommand\uf{\mathfrak{u}}
\newcommand\uno{\mathbf{1}}
\newcommand\vf{\mathfrak{v}}

\newcommand\wh{\widehat}
\newcommand\wt{\widetilde}
\newcommand\xf{\mathfrak{x}}
\newcommand\Xf{\mathfrak{X}}
\newcommand\yf{\mathfrak{y}}
\newcommand\zf{\mathfrak{z}}
\newcommand\Zf{\mathfrak{Z}}
\newcommand\Z{\mathbb Z}
\numberwithin{equation}{section}

\newtheoremstyle{mythm}% name
  {9pt}%      Space above, empty = `usual value'
  {9pt}%      Space below
  {\itshape}% Body font
  {0pt}%         Indent amount (empty = no indent, \parindent = paraindent)
  {\bfseries}% Thm head font
  {}%        Punctuation after thm head
  { }% Space after thm head: \newline = linebreak
  {\thmnumber{(#2)}\thmname{ #1}\thmnote{ #3}}%         Thm head spec

\newtheoremstyle{mydef}% name
  {9pt}%      Space above, empty = `usual value'
  {9pt}%      Space below
  {\normalfont}% Body font
  {0pt}%         Indensf suppt amount (empty = no indent, \parindent = paraindent)
  {\bfseries}% Thm head font
  {}%        Punctuation after thm head
  { }% Space after thm head: \newline = linebreak
  {\thmnumber{(#2)}\thmname{ #1}\thmnote{ #3}}%         Thm head spec

\theoremstyle{mythm}
\newtheorem{thm}[equation]{Theorem.}
\newtheorem{pro}[equation]{Proposition.}
\newtheorem{lem}[equation]{Lemma.}
\newtheorem{cor}[equation]{Corollary.}
\theoremstyle{mydef}

\newtheorem{rmk}[equation]{Remark.}

\begin{document}$\,$ \vspace{-1truecm}
\title{Brownian motion and Harmonic functions \\
on $\Sol(\pp,\qq)$}
\author{{\bf Sara BROFFERIO,
Maura SALVATORI, and Wolfgang WOESS}
\\ $\,$\\
{\normalfont Dedicated to Alessandro Fig\`a-Talamanca \\ on the occasion of his 70th
birthday}
}
\address{\parbox{.8\linewidth}{Sara Brofferio\\
Laboratoire de Math\'ematiques\\ 
Universit\'e de Paris-Sud b\^at 425\\  
91405 Orsay Cedex, France\\}}
\email{sara.brofferio@math.u-psud.fr}
\address{\parbox{.8\linewidth}{Maura Salvatori\\
Dipartimento di Matematica\\ 
Universit\`a degli Studi di Milano\\ 
Via Saldini, 50\\  
20133 Milano, Italy\\}}
\email{maura.salvatori@unimi.it}
\address{\parbox{.8\linewidth}{Wolfgang Woess\\
Institut f\"ur Mathematische Strukturtheorie,\\ 
Technische Universit\"at Graz,
Steyrergasse 30, 8010 Graz, Austria\\}}
\email{woess@TUGraz.at}\date{\today}
%\thanks{Supported by ESF program RDSES}

\begin{abstract}
The Lie group $\Sol(\pp,\qq)$ is the semidirect product 
induced by the action of $\R$ on $\R^2$ which is given by 
$(x,y) \mapsto (e^{\pp z}x,e^{-\qq z}y)$, $z \in \R$. Viewing 
$\Sol(\pp,\qq)$ as a 3-dimensional manifold, it carries a natural Riemannian metric
and Laplace-Beltrami operator. We add a linear drift term in the $z$-variable
to the latter, and study the associated Brownian motion with drift. We derive
a central limit theorem and compute the rate of escape. Also, we introduce
the natural geometric compactification of $\Sol(\pp,\qq)$ and explain how
Brownian motion converges almost surely to the boundary in the resulting
topology. We also study all positive harmonic functions for the Laplacian
with drift, and determine explicitly all minimal harmonic functions.
All this is carried out with a strong emphasis on understanding and using
the geometric features of $\Sol(\pp,\qq)$, and in particular the fact that 
it can be described as the horocyclic product of two hyperbolic planes with 
curvatures $-\pp^2$ and $-\qq^2$, respectively. 
\end{abstract}

\maketitle

\markboth{{\sf S. Brofferio, M. Salvatori, and W. Woess}}
{{\sf $\Sol(\pp,\qq)$}}
\baselineskip 15pt

%$\,$ \vspace{-1.5truecm}

\section{Introduction}\label{sect:intro}
$\Sol(\pp,\qq)$ is the group of all matrices of the form
\begin{equation}\label{eq:solmatrix}
\gf = \begin{pmatrix} e^{\pp z} & x & 0         \\
	0         & 1 & 0         \\
	0         & y & e^{-\qq z}
\end{pmatrix}\,,\quad x, y, z \in \R\,.
\end{equation}
The parameters $\pp$ and $\qq$ are positive real numbers.
It will be useful to think separately of $\Sol(\pp,\qq)$ as a Lie group and
as a manifold. In the latter situation, we shall often write $\zf = (x,y,z)$
or also $\xf$ or $\yf$ for its elements, instead of $\gf$.
Its length element is
$$
ds^2 = d_{\pp,\qq}s^2 = e^{-2\pp z}\,dx^2 + e^{2\qq z}\,dy^2 + dz^2\,,
$$
which is invariant under the left action of $\Sol(\pp,\qq)$
on itself as an isometry group. If we identify the element $\gf$ of
\eqref{eq:solmatrix} with $(x,y,z)$, then $\Sol(\pp,\qq)$ is $\R^3$
topologically (but of course not metrically). In those coordinates,
the group product is
\begin{equation}\label{eq:product}
(a,b,c) \cdot  (x,y,z) = \bigl( e^{\pp c}x + a, e^{-\qq c}y + b, c+z\bigr)\,.
\end{equation}
The purpose of this case study is to describe the behaviour of Brownian
motion in space and time, and to determine all positive harmonic functions
on $\Sol(\pp,\qq)$ with respect to its Laplace-Beltrami operator and
the variant where a ``vertical'' drift term (in $z$) is added to the latter.
More precisely, we shall derive a central limit theorem for Brownian motion
with drift, describe convergence of this process to the natural geometric
boundary at infinity, and we shall determine all positive eigenfunctions
of those Laplacians. The experienced reader will know how intimately such
stochastic and potential theoretic features are linked with each other.

Before we can explain the results, we need some details. 
%We have two natural projections onto hyperbolic space.
%
To start, let $\Hb = \{ x + \im w : x \in \R\,,\; w > 0 \}$ be
hyperbolic upper half plane with the standard length element
$w^{-2}(dx^2 + dw^2)$. We can pass to the logarithmic model by substituting
$z = \log w\,$, and in those coordinates the length element becomes
$e^{-2z}dx^2 + dz^2$. Now we also change curvature by modifying the length
element into
$$
ds^2 = d_{\pp}s^2 = e^{-2\pp z}\,dx^2+ dz^2\,.
$$
We write $\Hb(\pp)$ for the hyperbolic plane with this parametrization and
metric. Then we have the natural projections
\begin{equation}\label{eq:projections}
\begin{aligned}
&\pi_1: \Sol(\pp,\qq) \to \Hb(\pp)\,,\quad (x,y,z) \mapsto (x,z)\\
&\pi_2: \Sol(\pp,\qq) \to \Hb(\qq)\,,\quad (x,y,z) \mapsto (y,-z)\,.
\end{aligned}
\end{equation}
The \emph{horocycle at level $z$} in $\Hb(\pp)$ is the set
$\{(x,z) : x \in \R\}$, and we write $\hor(x,z) = z$. Thus, we get another
natural projection $\hor: \Hb(\pp) \to \R$. We also consider $\hor$ as a
projection of $\Sol(\pp,\qq)$ onto $\R$, where $\hor(x,y,z)=z$. We shall 
write $\dist$ for each of the metrics induced by the respective length
elments; it will usually be evident from the context to which of the
underlying spaces this refers -- or else, that space will appear in the index.
(On $\R$ we then have $\dist_{\R}(z_1,z_2) = |z_1-z_2|$.)
Note that our projections preserve distances in the following sense:
\begin{equation}\label{eq:preservedist}
\begin{aligned}
\dist_{\Sol}\bigl((x,y_1,z_1),(x,y_2,z_2)\bigr)
&= \dist_{\Hb(\qq)}\bigl((y_1,-z_1),(y_2,-z_2)\bigr)\,,\\
\dist_{\Sol}\bigl((x_1,y,z_1),(x_2,y,z_2)\bigr)
&= \dist_{\Hb(\pp)}\bigl((x_1,z_1),(x_2,z_2)\bigr)\,,\AND\\
\dist_{\Sol}\bigl((x,y,z_1),(x,y,z_2)\bigr) &= |z_1-z_2|\,.
\end{aligned}
\end{equation}
A main structural feature is that the manifold $\Sol(\pp,\qq)$ is made up
by two hyperbolic planes (with respective curvatures $-\pp^2$ and $-\qq^2$) that are
glued together by identifying opposite horocycles: it can be seen as the
\emph{horocyclic product} of $\Hb(\pp)$ and $\Hb(\qq)$,
\begin{equation}\label{eq:horpro}
\Sol(\pp,\qq) = \{ (\uf, \vf) \in \Hb(\pp) \times \Hb(\qq) :
\hor(\uf) + \hor(\vf) = 0\}\,,
\end{equation}
with its metric arising naturally from those two hyperbolic planes.

We remark here that there are various different types of horocyclic
products. $\Sol(\pp,\qq)$ has two sister structures. One is the
\emph{Diestel-Leader graph} $\DL(\pp,\qq)$, which is the horocyclic
product of two regular trees with degrees $\pp+1$ and $\qq+1$, respectively,
where $\pp, \qq \ge 2$ are integer. One of its interesting features is
that when $\pp=\qq$, it is a Cayley graph of the \emph{lamplighter group}
$(\Z/\pp\Z) \wr \Z$. Random walks and harmonic functions on $\DL(\pp,\qq)$ have
been studied intensively by {\sc Bertacchi}~\cite{Be},
{\sc Woess}~\cite{Wo-lamp}, {\sc Bartholdi and Woess}~\cite{BaWo} and
{\sc Brofferio and Woess}~\cite{BrWo1}, \cite{BrWo2}.
The other sister structure is
\emph{treebolic space} $\HT(\pp,\qq)$, which is the horocyclic product of
$\Hb(\pp)$ and the tree with degree $\qq+1$, where $\pp > 0$ (real) and
$\qq \ge 2$ (integer). When $\pp=\qq$, the Baumslag-Solitar group
$\langle a, b \mid ab = b^{\qq}a \rangle$ acts on $\HT(\qq,\qq)$ with
compact quotient. The study of potential theory and Brownian motion
on treebolic space is harder than on $\Sol$ and on $\DL$ (where random walk
replaces Brownian motion), first of all because of the conceptual and technical difficulty
in constructing the right Laplacian(s) on the 2-dimensional complex
$\HT$. This is ongoing work of {\sc Bendikov, Saloff-Coste, Salvatori and
Woess}~\cite{BSSW1}, \cite{BSSW2}.

Brownian motion and random walks on $\Sol(1,1)$ made a brief appearance
in the work of {\sc Lyons and Sullivan}~\cite{LySu}. Harmonic functions for random
walks on $\Sol(1,1)$ also appear in {\sc Raugi}~\cite[Exemple 2, p. 677]{Ra}.
%in a framework that applies to $\Sol(\pp,\qq)$ whenever $\pp=\qq$.

$\HT$, $\DL$ and $\Sol$ are also objects of great interest in relation
with geometric group theory. Quasi-isometries of those spaces have been
studied by {\sc Farb and Mosher}~\cite{FaMo} (for $\HT(\pp,\pp)$) and by
{\sc Eskin, Fisher and Whyte}~\cite{EFW1}, \cite{EFW2} (for $\DL$ and $\Sol$).
The last two papers also contain a good description of several aspects
of the geometry of $\Sol$.

The \emph{Laplace operator} with \emph{vertical drift parameter} $\ab \in \R$
on $\Sol(\pp,\qq)$ is
\begin{equation}\label{eq:lap-sol}
\Lap_{\ab} = \Lap_{\ab}^{\Sol(\pp,\qq)}
= \frac12\left(e^{2\pp z}\,\frac{\bd^2}{\bd x^2} + e^{-2\qq z}\,\frac{\bd^2}{\bd y^2}
+ \frac{\bd^2}{\bd z^2}\right) + \ab \,\frac{\bd}{\bd z}\,.
\end{equation}
The Laplace-Beltrami operator arises for $\ab = (\qq-\pp)/2$.

As a matter of fact, this involves a small abuse of terminology:
in differential geometry, the ``true'' Laplace-Beltrami operator
would be \emph{twice} the one which we are using.
Here, we are following the probabilistic habits: with the factor $\frac12$,
in the standard Euclidean situation, the Laplacian is the infinitesimal
generator of standard Brownian motion. The situation is similar here.
%, namely,
%with our normalisation, the projection $\hor$ will generate
%standard Brownian motion with drift $\ab$ in the $z$-variable.

Under the projection $\pi_1$, the operator $\Lap_a$ projects onto the
operator on $\Hb(\pp)$ given by
\begin{equation}\label{eq:lap1-hyp}
\Lap^{\Hb(\pp)}_{\ab} = \frac12\left(e^{2\pp z}\,\frac{\bd^2}{\bd x^2}
+ \frac{\bd^2}{\bd z^2}\right) + \ab \,\frac{\bd}{\bd z}\,.
\end{equation}
By ``projects'' we mean that for a $C^2$-function $f_1$ on $\Hb(\pp)$, one
(obviously) has $\Lap_{\ab} (f_1\circ\pi_1) = (\Lap^{\Hb(\pp)}_{\ab} f_1)\circ\pi_1\,$.
This is the Laplace-Beltrami operator on $\Hb(\pp)$ when $\ab = -\pp/2$.

Analogously, under the projection $\pi_2$ (where the sign of $z$ is changed),
$\Lap_{\ab}$ projects onto the operator on $\Hb(\qq)$ given by
\begin{equation}\label{eq:lap2-hyp}
\Lap^{\Hb(\qq)}_{-\ab} = \frac12\left(e^{2\qq z}\,\frac{\bd^2}{\bd y^2}
+ \frac{\bd^2}{\bd z^2}\right) - \ab \,\frac{\bd}{\bd z}\,.
\end{equation}

And finally, $\Lap_{\ab}$ projects under $\hor$ onto the operator on $\R$ given by
\begin{equation}\label{eq:lap-tilde}
\wt\Lap_{\ab} = \frac12\,\frac{d^2}{d z^2} + \ab \,\frac{d}{d z}\,.
\end{equation}

Coming back to the outline of the contents of this paper, some
basic preliminaries are laid out in \S \ref{sect:basic}.
Our first, probabilistic object of study is then Brownian motion with drift
$\Zf_t=(X_t, Y_t, Z_t)_{t > 0}$ on $\Sol(\pp,\qq)$, i.e., the diffusion
process whose infinitesimal generator is $\Lap_{\ab}\,$.
The projections of $\Zf_t$ on $\Hb(\pp)$, $\Hb(\qq)$
and $\R$ are the diffusions whose infinitesimal generators are
the respective projected operators defined above.

\smallskip

In \S \ref{sect:clt}, we describe this process in terms of stochastic
integrals and first derive a central limit theorem for $(X_t,Y_t, Z_t)$.
Combining this with estimates from \S 2 for the metric of $\Sol$,
we also obtain a central limit theorem for $\dist(\Zf_t, \of)$. Its form for
the case $\ab =0$ is somewhat different from what happens for $\ab \ne 0$.
As a corollary, we get the linear rate of escape:
$$
\frac{\dist(\Zf_t, \of)}{t} \to |\ab| \quad\text{almost surely, as\;}
t \to \infty\,,
$$
where $\of = (0,0,0)$.
This is the same as the rate of escape for the projected (``vertical'')
Brownian motion
with drift $(Z_t)_{t > 0}$ on $\R$, so that the lateral motion in the
$x$-and $y$-variables does not contribute to that rate.

\smallskip

Since the $\Sol$-group has exponential growth, our process is always transient,
that is, with probability $1$ it eventually leaves each compact set.
\S \ref{sect:boundary} adds more details to the description of how our process
tends to infinity in space. Namely, $\Sol(\pp,\qq)$ has a natural
geometric compactification: since $\Sol(\pp,\qq)$ is a subset of the product
of two hyperbolic planes (or equivalently, hyperbolic disks),
it embeds naturally into the product of two closed unit disks, and the
closure of $\Sol(\pp,\qq)$ in this
bi-disk is the compactification. Topologically, the resulting boundary at infinity
has the shape of
a filled number ``8'', that is, two full closed disks glued together at a 
single (glueing) point . 
It is not a ``visibility'' boundary:  neither the glueing point nor any of the
interior points of the two disks come up as a limit of some geodesic ray in $\Sol$; 
all the other boundary points \emph{are} limits of geodesic rays.

It is a rather straighforward, but nevertheless
informative task to verify that Brownian motion tends almost surely in the
topology of that compactification to a limit random variable that lives
on the boundary at infinity of $\Sol(\pp,\qq)$. If $\ab = 0$ then
this limit is the glueing point deterministically. Otherwise,
that limit random variable lies on one of the two circles that make
up the ``8'' (not their interiors) and its distribution is continuous. 
Thus, when $\ab \ne 0$,
we (almost surely) have the geodesic ray from the origin to the random 
limit point. If $\gab = \bigl(\gab(t)\bigr)_{t \ge 0}$ is that limit geodesic,
then we show that for $\ab \ne 0$, the deviation of $\Zf_t$ from that
ray is at most logarithmic, namely
$$
\limsup_{t \to \infty} \dist(Z_t, \gab)/\log t \le 2/|\ab| 
\quad\text{almost surely.}
$$
This result comprises the analogous one for Brownian motion with drift 
on the hyperbolic plane, where the bound is $1/|\ab|$. For the latter,
we are not aware of a proof that has appeared in print, but there is a correponding
theorem for random walks on free groups, resp. trees, that was 
first shown by {\sc Ledrappier}~\cite{Le}; see 
{\sc Woess}~\cite[Thm. 9.59]{Wo-markov} 
for a general and simple proof. 

\smallskip

The second main body of this
work concerns positive harmonic functions.
These are  the positive $C^2$-functions that are anihilated
by the respective Laplacian. We can also handle positive eigenfunctions.

We start in \S \ref{sect:potential} by displaying some of the potential
theoretic, resp. analytic ingredients that are needed. Then we prove
in \S \ref{sect:main} that
every positive $\Lap_{\ab}$-eigenfunction on $\Sol(\pp,\qq)$ has the
form
$$
h(x,y,z) = h_1(x,z) + h_2(y,-z)\,,
$$
where $h_1$ is a non-negative $\Lap^{\Hb(\pp)}_{\ab}$-eigenfunction on
$\Hb(\pp)$ and
$h_2$ is non-negative $\Lap^{\Hb(\qq)}_{-\ab}$-eigenfunction on $\Hb(\qq)$,
both with the same eigenvalue as $h$.

This decomposition is not unique, but we can also see
where non-uniqueness comes from, namely, harmonic functions that only depend on
the ``height'' $z$.
What we do is indeed to describe all minimal positive eigenfunctions,
based on ideas from the discrete setting of Diestel-Leader graphs,
see \cite{BrWo2}.

Since the positive $\Lap_{\ab}^{\Hb(\pp)}$-eigenfunctions are known
explicitly as integrals of modified Poisson kernels,
the above result leads to
a complete description of all positive $\Lap_{\ab}$-eigenfunctions
on $\Sol(\pp,\qq)$.
Thus, the positive eigenfunctions of the Laplacian on $\Sol(\pp,\qq)$
can be described fully in terms of (modified) Poisson kernels on
each of the two hyperbolic planes that make up our space.

The computations undertaken here are related with the study of Martin
compactifications of symmetric spaces, although the group $\Sol(\pp,\qq)$
embeds into this context only when $\pp=\qq$. The reader is referred to
the book by {\sc Guivarc'h, Ji and Taylor}~\cite{GJT} and the survey
by {\sc Kaimanovich}~\cite{Ka} plus the references given there. In particular,
we get close to answering the question of {\sc Lyons and Sullivan}~\cite{LySu}
to determine the Martin boundary of $\Sol$; we find the minimal boundary
and have a clear idea what the Martin compactification has to be.

We want to underline that the main spirit of this paper is to
study the outlined issues via strong use of the geometry of
$\Sol(\pp,\qq)$ in terms of the two hyperbolic planes and their
horocyclic product.
\\[5pt]
{\bf Acknowledgements.}
We warmly thank M. Yor, who suggested to the first of the three of us
the approach used for proving the central limit via stochastic
integration. Our own approach would have been random walk based, following
the spirit of {\sc Grincevi\v cjus}~\cite{Gr}. Also, we thank A. Grigoryan
for precious input regarding the Harnack inequality used in
Proposition \ref{pro:harnack}. Finally, we acknowledge instructive
hints by V. A. Kaimanovich concerning the literature.

\section{Basic facts}\label{sect:basic}

The first part of this section contains some basic facts regarding
$\Sol(\pp,\qq)$ that are quite straightforward. They are included here
for the sake of the completeness of the picture; most proofs are omitted.

\begin{lem}\label{lem:HaarSol} 
The Riemannian volume element of the $\Sol$-manifold is 
$$
d\zf = e^{(\qq-\pp)z}\, dx \, dy \, dz\,.
$$
This is also the left Haar measure of $\Sol(\pp,\qq)$ as a group. 
The modular function on this group is
$\;\Delta(\gf) = e^{(\qq-\pp)\,\hor(\gf)}\,$,
where $\gf$ is parametrized by $(x,y,z)$ as in \eqref{eq:solmatrix} and
$\hor(\gf)=z$. The group is unimodular if and only if $\pp = \qq$.
\end{lem}

% x % \begin{proof}
% x % Let $f$ be a compactly supported, continuous function on $\Sol$ (i.e.,
% x % a function on $\R^3$). Let $\gf_0 = (a,b,c)$ be a group element.
% x % Then, using \eqref{eq:product}
% x % $$
% x % \int f(\gf_0\gf)\, d_l\gf =
% x % \int\!\!\!\!\int\!\!\!\!\int
% x % f\bigl( e^{\pp c}x + a, e^{-\qq c}y + b, c+z\bigr)\,
% x %                       e^{(\qq-\pp)z}\, dx \, dy \, dz\,.
% x % $$
% x % Substituting $\xi = e^{\pp c}x + a$, $\eta = e^{-\qq c}y + b$ and
% x % $\zeta = c+z$, this becomes
% x % $$
% x % \int\!\!\!\!\int\!\!\!\!\int f( \xi,\eta,\zeta\bigr)\,
% x %                       e^{(\qq-\pp)\zeta}\, d\xi \, d\eta \, d\zeta
% x % = \int f(\gf)\, d_l\gf \,.
% x % $$
% x % This proves (a).
% x % Regarding (b), we have to compute
% x % $$
% x % \int f(\gf_0\gf)\, d_r\gf =
% x % \int\!\!\!\!\int\!\!\!\!\int f\bigl( e^{\pp z}a + x, e^{-\qq z}b + y, c+z\bigr)\,
% x %                       dx \, dy \, dz\,,
% x % $$
% x % which coincides with $\int f(\gf)\, d_r \gf$ by translation invariance of
% x % Lebesgue measure, so that $d_r \gf$ is indeed right Haar measure.
% x % Statement (c) now follows from the defining relation
% x % $\int f(\gf\gf_0)\,d_l \gf = \Delta(\gf_0^{-1}) \int f(\gf)\,d_l\gf\,$ for the
% x % modular function.
% x % \end{proof}

Next, consider the group $\Aff(\pp)$ of all matrices of the form
\begin{equation}\label{eq:affmatrix}
\begin{pmatrix} e^{\pp z} & x         \\
	0         & 1
\end{pmatrix}\,,\quad x, z \in \R\,.
\end{equation}
This is nothing but the group of orientation preserving \emph{affine
transformations} of hyperbolic plane, again parametrized by the
logarithmic model and substituting the habitual upper left term  $e^z$
with $e^{\pp z}$. We can identify the group $\Aff(\pp)$ with the surface %manifold
$\Hb(\pp)$ in the same way as we identified $\Sol$ as a group with
$\Sol$ as a manifold. By left multiplication, $\Aff(\pp)$ acts isometrically
on $\Hb(\pp)$. We recall the following. 

\begin{lem}\label{lem:HaarAff} \emph{(a)}
The Riemannian area element of $\Hb(\pp)$ is $e^{-\pp z}\, dx \, dz\,$.
This is also the left Haar measure on the group $\Aff(\pp)$, and the modular 
function on $\Aff(\pp)$ is
$\Delta(\gf) = e^{\pp\,\hor(\gf)}\,,$
where
$\gf =\bigl(\begin{smallmatrix} e^{\pp z} & x \\[1pt] 0 & 1\end{smallmatrix}\bigr)$
and $\hor(\gf)=z$.
\end{lem}

We can interpret the projections $\pi_1$ and $\pi_2$ of \eqref{eq:projections}
as homomorphisms from the group $\Sol(\pp,\qq)$ onto $\Aff(\pp)$ and
$\Aff(\qq)$, respectively. In the same way, $\hor$ is a homomorphism onto
the additive group $\R$.

\begin{lem}\label{lem:reversible}
\emph{(a)} The Laplacian $\Lap_{\ab}$ on $\Sol(\pp,\qq)$ is reversible
(self-adjoint) with respect to the measure
$$
\ms_{\ab}(d\zf)= e^{(2\ab + \pp - \qq)z}\,d\zf = e^{2\ab z} \, dx\,dy\,dz\,  .
$$
\emph{(b)} The Laplacian $\Lap^{\Hb(\pp)}_{\ab}$ on $\Hb(\pp)$ is reversible
with respect to the measure
$$
e^{2\ab z} \, dx\,dz\,.
$$
\emph{(c)} The Laplacian $\wt\Lap_{\ab}$ on $\R$ is reversible
with respect to the measure
$$
e^{2\ab z} \, dz\,.
$$
\end{lem}

\begin{proof}[Proof (hint)] For proving (a) one has to show that for
compactly supported
$C^2$-functions $f, g$ on $\Sol$, one has
$$
\int\!\!\!\!\int\!\!\!\!\int f(x,y,z)\,\,\Lap_{\ab}g(x,y,z)\, e^{2\ab z}
\, dx\,dy\,dz =
\int\!\!\!\!\int\!\!\!\!\int \Lap_{\ab}f(x,y,z)\,\,g(x,y,z)\, e^{2\ab z}
\, dx\,dy\,dz\,.
$$
This is straightforward by partial integration. (b) and (c) are analogous.
\end{proof}

Our Laplacian is invariant under the group action
of $\Sol$. Let $\gf_0 = (a,b,c)$ be a group element, and define the translate
of a function $f$ on $\Sol$ as
$\tau_{\gf_0}f(\gf) = f(\gf_0\gf)\,$, that is,
\begin{equation}\label{eq:gtranslate}
\tau_{\gf_0}f(x,y,z) = f\bigl( e^{\pp c}x + a, e^{-\qq c}y + b, c+z\bigr)
\,.
\end{equation}

\begin{lem}\label{lem:lap-inv} For any $\gf_0 \in \Sol(\pp,\qq)$,
$$
\Lap_{\ab}(\tau_{\gf_0}f) = \tau_{\gf_0}(\Lap_{\ab}f)\,.
$$
\end{lem}

The proof is completely elementary, using \eqref{eq:gtranslate}.

\smallskip

We shall need the following observations on the metric.
Regarding our hyperbolic planes  in the logarithmic model, let us remark
here that the metric of $\Hb(\pp)$ is linked with the standard one 
of $\Hb= \Hb(1)$ by the formula
\begin{equation}\label{eq:metric-p-1}
\dist_{\Hb(\pp)}\bigl((x,z)\,,\,(x',z')\bigr)
= \frac{1}{\pp}\,\dist_{\Hb(1)}\bigl((\pp x,\pp z)\,,\,(\pp x',\pp z')\bigr).
\end{equation}

%% 
%% If $\gab(t) =\bigl(x(t), z(t)\bigr)$, $t \in [a\,,\,b]$ is any piecewise 
%% differentiable curve in $\Hb(\pp)$ then its length $l_{\Hb(\pp)}(\gab)$ in 
%% $\Hb(\pp)$ satisfies the formula
%% $$
%% l_{\Hb(\pp)}(\gab) = l_{\H(1)}(\pp\cdot\gab)\,,  
%% $$
%% where $\pp\cdot\gab(t) =\bigl(\pp\, x(t),\pp\, z(t)\bigr)$.
%% Now recall that the unique geodesic arc between any pair of points in $\H(1)$
%% is such that in the $(x,e^z)$-half plane it lies on an orthogonal
%% half-circle or half-line to the $x$-axis. Therefore, 

While for Diestel-Leader graphs, there is an explicit formula for
the graph metric in terms of the two underlying trees \cite{Be},
we do not have such a formula on $\Sol$. However, we have at least
the following distance estimates.

\begin{pro}\label{pro:metric} For all 
$\zf = (x,y,z)\in \Sol(\pp,\qq)$, with $x,y \ne 0$ 
%and constants $c_{\pp}$, $c_{\qq} > 0$ 
in \emph{(iv)}, 
\begin{align}
\dist_{\Sol}(\of, \zf) &\ge |z|\,, \tag{i}\\
\dist_{\Sol}(\of, \zf) 
&\ge 2 \,\frac{\log|x|}{\pp} +  2 \,\frac{\log|y|}{\qq}
-|z| - \left(\frac{1}{\pp} + \frac{1}{\qq}\right) 
\log \dist_{\Sol}(\of, \zf) \,,\tag{ii}\\
\dist_{\Sol}(\of, \zf) 
&\le \dist_{\Hb(\pp)}\bigl((x,z)\,,\, (0, 0)\bigr)
\,+\, \dist_{\Hb(\qq)}\bigl((y,-z)\,,\, (0,0)\bigr)
- |z|\tag{iii}\\
&\le c + \frac{2\log(1+|x|)}{\pp} +\frac{2\log(1+|y|)}{\qq}+|z|\,,\notag\\ 
\dist_{\Sol}(\of, \zf)
&\le c' + \left|\frac{\log|x|}{\pp} 
+ \frac{\log|y|}{\qq}\right|\tag{iv}\\
&\qquad+ \min \left\{ \left| \frac{\log|x|}{\pp} \right| +
\left| \frac{\log|y|}{\qq} + z \right|\,,\,   
\left| \frac{\log|x|}{\pp} -z \right| +
\left| \frac{\log|y|}{\qq}\right|\right\},\notag
\end{align}
where $c, c' > 0$.
\end{pro}

\begin{proof}
Inequality (i) is clear.

\smallskip

For (ii), Let $\zf(t)=\bigl(x(t),y(t),z(t)\bigr)_{t\in[0, d]}$ be a geodesic 
path in $\Sol$ from $\of$ to $\zf$, where $d=\dist_{\Sol}(\of , \zf)$. 
Let 
$$ 
M=\max \{ z(t) :t\in[0, d]\} \AND m= \min\{ z(t) :t\in[0, d]\}\,,
$$
so that $M \ge 0$ and $m \le 0$. Then
$$
\begin{aligned}
\dist_{\Sol}(\zf\,,\,\of)&=
\int_0^d\sqrt{e^{-2\pp z(t)}\dot{x}(t)^2+e^{2\qq z(t)}\dot{y}(t)^2+\dot{z}(t)^2}
\,dt\\
&\ge e^{-\pp M}\int_0^d\sqrt{\dot{x}(t)^2+\dot{z}(t)^2}\,dt
\ge e^{-\pp M} \sqrt{x^2+z^2}
\end{aligned}
$$
Thus
$$ 
\pp  M\ge \log \sqrt{x^2+z^2} - \log \dist_{\Sol}(\of,\zf)\,,\; 
\text{and analogously}\;
-\qq m\ge \log \sqrt{y^2+z^2} - \log \dist_{\Sol}(\of,\zf)\,.
$$
Now let $\zf_M$ and $\zf_m$ be points on the geodesic from 0 to $\zf$ with 
heights $M$ and $m$, respectively. Then, according to which of the two
``comes first'', (i) yields that
either
$$
\begin{aligned}
\dist_{\Sol}(\of,\zf) &= \dist_{\Sol}(\of,\zf_M)+ \dist_{\Sol}(\zf_M, \zf_m)+\dist_{\Sol}(\zf_m,\zf)
\ge M + (M-m) + (z-m)\,,\quad\text{or}\\
\dist_{\Sol}(\of,\zf) &= \dist_{\Sol}(\of,\zf_m)+ \dist_{\Sol}(\zf_m, \zf_M)+\dist_{\Sol}(\zf_M,\zf)
\ge -m + (M-m) + (M-z)\,.
\end{aligned}
$$
We see that 
$
\dist_{\Sol}(\of,\zf) \ge 2(M-m) - |z|\,,
$
and combining this with the above, we obtain (ii). 

\smallskip

For proving the first part of (iii), we may suppose without loss of generality that 
$z \ge 0$. 

Note that in the logarithmic model of $\Hb(\pp)$, 
%in the logarithmic model, 
any geodesic arc is either vertical (i.e., of the form
$t \mapsto (x_0,t)$, where $x_0$ is fixed and $t$ varies in an interval),
or else it can be realised as $t \mapsto \bigl(t,z(t)\bigr)$, where $z(t)$ is a strictly 
concave function of $t$ varying in an interval. 

Let $(x',z)$ be the first (``leftmost'') point on the geodesic arc 
from $(0,0)$ to $(x,z)$ in $\Hb(\pp)$ with second coordinate $z$, 
and  let $(y',0)$ be the last (``rightmost'')
point on the geodesic arc from $(0,0)$ to $(y,-z)$ in $\Hb(\qq)$
with second coordinate $0$. 
We may have $x'=x$ or $y'=0$, but in any case, the geodesic arc from 
$(0,z)$ to $(x',z)$ in $\Hb(\pp)$ is strictly increasing in both
coordinates, while the geodesic arc from $(y',0)$ to $(y,-z)$ in $\Hb(\qq)$
is strictly inreasing in the first and strictly decreasing in the second
variable.
That is, these two arcs can be parametrised, respectivley,  as
$$
t \mapsto \bigl( x(t),t \bigr) \AND t \mapsto \bigl(y(t),-t\bigr)\,,
$$
where $t \in [0\,,\,z]$ and $\dot x(t), \dot y(t) > 0$.
Now we can ``synchronise'' the two in order
to get the curve 
$$
t \mapsto \bigl( x(t),y(t), t \bigr)\,,\quad  t \in [0\,,\,z]\,,
$$
that connects $(0,y',0)$ with $(x',y,z)$ in $\Sol(\pp,\qq)$. 
The length of this curve majorises the distance between these two points in 
$\Sol(\pp,\qq)$ and is
$$
\begin{aligned}
\int_{0}^{z} 
&\sqrt{e^{-2\pp t}\, \dot x(t)^2 + e^{2\qq t} \,\dot y(t)^2 +1}\,\,dt\\
&\le \int_{0}^{z} \left(\sqrt{e^{-2\pp t}\, \dot x(t)^2  + 1}
+ \sqrt{e^{2\qq t}\, \dot y(t)^2 + 1} - 1\right)dz\\
&= \dist_{\Hb(\pp)}\bigl((0,0)\,,\, (x', z)\bigr)
+ \dist_{\Hb(\qq)}\bigl((y',0)\,,\, (y,-z)\bigr)
- z\,.
\end{aligned}
$$
Now, by  \eqref{eq:preservedist}, 
%with $\dist = \dist_{\Sol}$,
$$
%\begin{aligned}
%&
\dist_{\Sol}(\of\,,\, \zf) %\\
%&\;
\le \underbrace{\dist_{\Sol}\bigl((0,0,0)\,,\, (0, y',0)\bigr)}_{ 
\displaystyle =\dist_{\Hb(\qq)}\bigl((0,0)\,,\, (y',0)\bigr)}
+\, \dist_{\Sol}\bigl((0,y',0)\,,\, (x', y,z)\bigr) 
+ \underbrace{\dist_{\Sol}\bigl((x',y,z)\,,\, (x, y,z)\bigr)}_{ 
\displaystyle =\dist_{\Hb(\pp)}\bigl((x',z)\,,\, (x,z)\bigr)}
%\end{aligned}
$$
We insert the upper bound for the middle term that we derived above.
Since 
$$
\begin{aligned}
\dist_{\Hb(\pp)}\bigl((0,0)\,,\, (x', z)\bigr)
+ \dist_{\Hb(\pp)}\bigl((x',z)\,,\, (x,z)\bigr) &=
\dist_{\Hb(\pp)}\bigl((0,0)\,,\, (x, z)\bigr)\AND\\
\dist_{\Hb(\qq)}\bigl((0,0)\,,\, (y',-z)\bigr)
+ \dist_{\Hb(\qq)}\bigl((y',-z)\,,\, (y,-z)\bigr) &=
\dist_{\Hb(\qq)}\bigl((0,0)\,,\, (y,-z)\bigr),
\end{aligned}
$$
the proposed inequality follows. For the second part of (iii), we use
\eqref{eq:metric-p-1}:
$$
\begin{aligned}
\dist_{\Hb(\pp)}\bigl((x,z)\,,\, (0, 0)\bigr) 
&\le |z| + \frac{1}{\pp}\dist_{\Hb(1)}\bigl((\pp x,0)\,,\, (0, 0)\bigr) 
= |z| +\frac{1}{\pp} \log \frac{\sqrt{ (\pp x)^2 + 4 } + |\pp x|}
{\sqrt{ (\pp x)^2 + 4 } - |\pp x|} \\
&\le |z| + \frac{1}{\pp}\log \bigl((\pp x)^2 + \pp|x| + 1 \bigr) 
\le |z| + \frac{2\log(1+\pp)}{\pp}  + \frac{2\log(1+|x|)}{\pp} \,.
\end{aligned}
$$
Combining this with the analogous bound for 
$\dist_{\Hb(\qq)}\bigl((y,-z)\,,\, (0,0)\bigr)$, the inequality follows.
%and $c = \frac{2\log(1+\pp)}{\pp} + \frac{2\log(1+\qq)}{\qq}\,$.

\smallskip

For proving (iv), first note that for all $x \ne 0$,
$$
\dist_{\Hb(\pp)}\Bigl(\bigl( 0, \tfrac{\log|x|}{\pp}\bigr)\,,\,
  \bigl( x, \tfrac{\log|x|}{\pp}\bigr)\Bigr) =  c_{\pp}
$$
depends only on $\pp$. 
Then, using \eqref{eq:preservedist},
$$
\begin{aligned}
\dist_{\Sol}(\of\,,\, \zf) 
&\le \dist_{\Sol}\Bigl(\of\,,\, \bigl( 0,0, \tfrac{\log|x|}{\pp}\bigr)\Bigr)  
+ \dist_{\Sol}\Bigl(\bigl( 0,0, \tfrac{\log|x|}{\pp}\bigr)\,,\,
		    \bigl( x,0, \tfrac{\log|x|}{\pp}\bigr)\Bigr)\\
&\quad+ \dist_{\Sol}\Bigl(\bigl( x,0, \tfrac{\log|x|}{\pp}\bigr)\,,\,
		    \bigl( x,0, -\tfrac{\log|y|}{\qq}\bigr)\Bigr)\\
&\quad+  \dist_{\Sol}\Bigl(\bigl( x,0, -\tfrac{\log|y|}{\qq}\bigr)\,,\,
		    \bigl( x,y, -\tfrac{\log|y|}{\qq}\bigr)\Bigr)
+ \dist_{\Sol}\Bigl(\bigl( x,y, -\tfrac{\log|y|}{\qq}\bigr)\,,\,
		    \zf\Bigr)\\
&= \Big|\tfrac{\log|x|}{\pp}\Big| + c_{\pp}  
+ \Big|\tfrac{\log|x|}{\pp}+ \tfrac{\log|y|}{\qq}\Big|+ c_{\qq}
+ \Big|\tfrac{\log|y|}{\qq}+z\Big|\,.	    
\end{aligned} 		    
$$
Exchanging the roles of $x$ and $y$, as well as of $\pp$ and $\qq$, the inequality
follows.
\end{proof}

%\section{Geometric boundary, geodesics, and regular trajectories}\label{sect:clt}

\section{Central limit theorem and rate of escape}\label{sect:clt}

Let $\Zf_t=(X_t,Y_t , Z_t)$, $t \ge 0$, be the continuous diffusion on
$\Sol(\pp,\qq) \equiv \R^3$ whose infinitesimal generator is $\Lap_\ab$.
If the starting point is $\of = (0,0,0)$, then  $\Zf_t$ is given
by the stochastic integrals
\begin{equation}\label{eq:stochint}
\left \{ \begin{array}{l}
\displaystyle Z_t= \ab\, t + W_t\,,\\[12pt]
\displaystyle X_t= \int_0^t e^{\pp Z_s}\, dW^{(1)}_s\\[12pt]
\displaystyle Y_t= \int_0^t e^{-\qq Z_s}\, dW^{(2)}_s\,,
\end{array}\right.
\end{equation}
where  $(W_t, W_t^{(1)},W_t^{(2)})_{t \ge 0}$ are three independent standard
Brownian motions. (We do not
attach a superscript to the one that defines the coordinate $Z_t$, because
this is the most important one that determines the behaviour of all three.)
See for instance {\sc Revuz and Yor}~\cite{RY} or {\sc Protter}~\cite{Pr},
and compare, in particular, with
{\sc Baldi, Casadio Tarabusi, Fig\`a-Talamanca and Yor}~\cite{BCFY}.

For the following central limit theorem, let
$$
\mathcal N = W_1\,, \quad \underline{\mathcal M} = \min \{ W_t : 0 \le t \le 1\}
\AND \overline{\!\mathcal M} = \max \{ W_t : 0 \le t \le 1\}\,,
$$
so that $\mathcal N$ has standard normal distribution.

\begin{thm}\label{thm:clt} \emph{(i)} If $\ab >0$, then as $t \to \infty$
$$
\frac{1}{\sqrt{t}}\Bigl(\log|X_t|-\pp \ab \, t\,,\,
\log|Y_t|\,,\, Z_t-\ab \,t\Bigr)
\to (\pp\, \mathcal N\,,\,0\,,\,\mathcal N) \quad \text{in law.}
$$
\emph{(ii)} If $\ab < 0$, then as $t \to \infty$
$$
\frac{1}{\sqrt{t}}\Bigl(\log|X_t|\,,\,
\log|Y_t|+\qq \ab\, t \,,\, Z_t-\ab \, t\bigr)
\to (0\,,\,\qq\, \mathcal N\,,\,\mathcal N)
\quad \text{in law.}
$$
\emph{(iii)} If $\ab =0$, then as $t \to \infty$
$$
\frac{1}{\sqrt{t}}\Bigl(\log|X_t|\,,\,\log|Y_t|\,,\, Z_t\bigr)
\to (\pp\,\overline{\!\mathcal M}\,,\,-\qq\, \underline{\mathcal M}
\,,\,\mathcal N)
\quad \text{in law.}
$$
\end{thm}

\begin{proof}
For $\alpha \in \R$, set
$$
V_t(\alpha) = \int_0^t e^{2\alpha Z_s} \,ds,
$$
so that the quadratic variations of $X_t$ and $Y_t$ are
$V_t(\pp)$ and  $V_t(-\qq)$, respectively.
Then by a theorem of {\sc Dambis, Dubin and Schwartz}~\cite{Da}, \cite{DS},
see also \cite[p. 173]{RY}, there exist two standard Brownian motions $(B_t^{(1)})_{t\ge 0}$ and
$(B_t^{(2)})_{t\ge 0}$ such that
\begin{equation}\label{eq:XYB}
X_t = B_{V_t(\pp)}^{(1)} \AND Y_t = B_{V_t(-\qq)}^{(2)}.
\end{equation}
By  a theorem of {\sc Knight}~\cite{Kn}, see \cite[p. 175]{RY},
the processes $(B_t^{(1)})_{t\ge 0}$ and $(B_t^{(2)})_{t\ge 0}$ are independent
in our case.

By the scaling property of Brownian motion, for $i=1,2$ and $\alpha = \pp$,
resp. $\alpha =-\qq$,
$$
\frac{\log \,\bigl|\,B^{(i)}_{V_t(\alpha)}\big/\sqrt{V_t(\alpha)}\,\bigr|}{\sqrt{t}}
=\frac{\log\left|B_1\right|}{\sqrt{t}}\to 0 \quad \text{in law.}
$$
In the following computations we use frequently the following simple fact.
\begin{equation}\label{eq:fact}
\text{If $\;A_t\to A\;$ and
$\;C_t\to 0\;$ in law then  $\;(A_t,C_t)\to (A,0)\;$ in law, as $\;t\to
\infty\,$.}
\end{equation}
Thus
\begin{equation}\label{eq:varid}
\begin{aligned}
&\lim_{t\to\infty}\frac{1}{\sqrt{t}}
\Bigl(\log|X_t|-\pp\,\ab\, t\,,\,\log|Y_t|\,,\, Z_t-\ab \,t\Bigr)
\\
&\qquad=\lim_{t\to\infty}\frac{1}{\sqrt{t}}
\Bigl(\log \sqrt{V_t(\pp)} - \pp \,\ab\, t\,,\,
\log \sqrt{V_t(-\qq)}\,,\, W_t\Bigr)
\quad\text{in law.}
\end{aligned}
\end{equation}

\noindent
Case (i): $\ab >0$.
First observe that for all $\alpha>0$ and $\beta \in \R$,
$$
\lim_{t\to\infty} \int_0^t e^{-\alpha s+\beta W_s}\, ds
=\int_0^\infty e^{-\alpha s+\beta W_s}\, ds
\in (0,+\infty) \quad \text{almost surely,}
$$
since $\lim_{s \to \infty} \log (e^{-\alpha s+\beta W_s})/s=
-\alpha<0$. Therefore
\begin{equation}\label{eq:Vtneg}
\frac{1}{\sqrt{t}}\log \sqrt{V_t(-\alpha)} \to 0 \quad \text{in law, when}\; \alpha > 0.
\end{equation}
Using \eqref{eq:fact}, we get (in law) that
$$
\begin{aligned}\lim_{t\to\infty}
&\frac{1}{\sqrt{t}}\Bigl( \log \sqrt{V_t(\pp)} - \,\pp\,\ab\,t\,,\,
\log \sqrt{V_t(-\qq)}\,,\, W_t\Bigr)\\
&\qquad=\lim_{t\to\infty}
\frac{1}{\sqrt{t}}\Bigl( \log \sqrt{V_t(\pp)} - \,\pp\,\ab\,t\,,\,
0\,,\,W_t \Bigr)\\
&\qquad=\lim_{t\to\infty}
\frac{1}{\sqrt{t}}\left(\frac12
\log \Bigl(\int_0^t e^{2\pp(\ab(s-t)+W_s)}\,ds\Bigr)\,,\,
0\,,\,W_t \right)\\
&\qquad=\lim_{t\to\infty}
\frac{1}{\sqrt{t}}\left(\frac12
\log \Bigl(e^{2\pp W_t}\int_0^t e^{2\pp(\ab(s-t)+W_s-W_t)}\,ds\Bigr)\,,\,
0\,,\,W_t \right)
\end{aligned}
$$
Since $(W_{t-s})_{s\leq t} = (W_t-W_s)_{s\le t}$ in law, 
$$
\begin{aligned}
&\lim_{t\to\infty}
\frac{1}{\sqrt{t}}
\log \Bigl(\int_0^t e^{2\pp(\ab(s-t)+W_s-W_t)}\,ds\Bigr)=\lim_{t\to\infty}
\frac{1}{\sqrt{t}}
\log \Bigl(\int_0^t e^{-2\pp(\ab(t-s)-W_{t-s})}\,ds\Bigr)\\
&\qquad=\lim_{t\to\infty}
\frac{1}{\sqrt{t}}
\log \Bigl(\int_0^t e^{-2\pp(\ab(s)-W_{s})}\,ds\Bigr)=0\\
\end{aligned}
$$
and using
\eqref{eq:fact} once more, we get that 
$$
\lim_{t\to\infty}\frac{1}{\sqrt{t}}
\Bigl(\log|X_t|-\pp\,\ab\, t\,,\,\log|Y_t|\,,\, Z_t-\ab \,t\Bigr)
=
\lim_{t\to\infty}\frac{1}{\sqrt{t}}\Bigl(\pp\,W_t\,,\,0\,,\,W_t \Bigr)
= (\pp\,\mathcal N\,,\,0\,,\,\mathcal N)
$$
in law.

\smallskip

\noindent
Case (ii): $\ab <0$. This is obtained from Case (i) by exchanging the roles
of the $x$- and $y$-coordinates.
\\[5pt]
Case (iii): $\ab=0$. We take up \eqref{eq:varid} and continue to compute, with
all identities holding in law
$$
\begin{aligned}
\lim_{t\to\infty}&\,\frac{1}{\sqrt{t}}
\Bigl(\log \sqrt{V_t(\pp)}\,,\,  \log \sqrt{V_t(-\qq)}\,,\, W_t\Bigr)\\
&=
\lim_{t\to\infty}\frac{1}{2\sqrt{t}}
\left( \log \Bigl(\int_0^t e^{2\pp W_s}\,ds\Bigr)\,,\,
\log \Bigl(\int_0^t e^{-2\qq W_s}\,ds\Bigr)\,,\, 2W_t\right)\\
&=
\lim_{t\to\infty}\frac{1}{2\sqrt{t}}
\left( \log \Bigl(t \int_0^1 e^{2\pp W_{st}}\,ds\Bigr)\,,\,
\log \Bigl(t\int_0^1 e^{-2\qq W_{st}}\,ds\Bigr)\,,\, 2W_t\right)\\
&=
\lim_{t\to\infty}\frac{1}{2\sqrt{t}}
\left( \log \Bigl(t \int_0^1 e^{2\pp \sqrt{t}\,W_{s}}\,ds\Bigr)\,,\,
\log \Bigl(t\int_0^1 e^{-2\qq \sqrt{t}\,W_{s}}\,ds\Bigr)\,,\, 2\sqrt{t}\,W_1
\right)\\[4pt]
\noalign{\noindent[setting$\; \tau = 2\sqrt{t}\,$]}
&=
\lim_{\tau\to\infty}
\left( \log \Bigl(\int_0^1 e^{\tau \pp W_s}\,ds\Bigr)^{1/\tau}\,,\,
\log \Bigl(\int_0^1 e^{-\tau\qq W_s}\,ds\Bigr)^{1/\tau}\,,\, W_1\right)
\end{aligned}
$$
since $(W_{st})_{0\leq s\leq 1}=(\sqrt{t}\,W_s)_{0\leq s\leq 1}$ in law.
Now recall that the $L^{\tau}$-norm on $C([0\,,\,1])$ converges to the $L^{\infty}$-norm
as $\tau \to \infty$. We apply this to the functions
$s \mapsto  e^{\pp W_s}$ and $s \mapsto  e^{-\qq W_s}$, respectively, and then
take logarithms. Thus, almost surely
$$
\begin{aligned}
&\log \Bigl(\int_0^1 e^{\tau\pp W_s}\,ds\Bigr)^{1/\tau}
\to \pp\, \max\{ W_s : 0 \le s \le 1\} \AND\\
&\log \Bigl(\int_0^1 e^{-\tau\qq W_s}\,ds\Bigr)^{1/\tau}
\to -\qq\, \min\{ W_s : 0 \le s \le 1\}.
\end{aligned}
$$
This leads to statement (iii).
\end{proof}

Next, with $\mathcal N$, $\underline{\mathcal M}$
and $\overline{\!\mathcal M}$ as above, we deduce the following 
central limit theorem for the distance of Brownian
motion to the origin.

\begin{thm}\label{thm:clt-dist} If $\ab \ne 0$ then 
$$
\frac{\dist_{\Sol}(\Zf_t\,,\,\of) - |\ab|\,t}{\sqrt{t}} \to \mathcal N
\quad\text{in law, as} \;t \to \infty\,.
$$
If $\ab = 0$ then 
$$
\frac{\dist_{\Sol}(\Zf_t\,,\,\of)}{\sqrt{t}} 
\to 2(\,\overline{\!\mathcal M} - \underline{\mathcal M}\,) - |\mathcal N|
\quad\text{in law, as} \;t \to \infty\,.
$$ 
%as $t \to \infty$.
\end{thm}

\begin{proof}
We start with $\ab > 0$. Combining Theorem \ref{thm:clt}(i) with
Proposition \ref{pro:metric}(iv), we obtain in law
$$
\begin{aligned}
\lim_{t\to\infty}\, \frac{\dist_{\Sol}(\Zf_t\,,\,\of) - \ab\,t}{\sqrt{t}}
&\le  \lim_{t\to\infty}
\frac{\left|\frac{\log|X_t|}{\pp}+\frac{\log|Y_t|}{\qq}\right|
  + \left|\frac{\log|Y_t|}{\qq}\right|
  + \left|\frac{\log|X_t|}{\pp}-Z_t\right| -\ab t}{\sqrt{t}} \\
&= \lim_{t\to\infty}\frac{\left|\frac{\log| X_t|}{\pp} \right| -\ab t 
  + \left|\bigl(\frac{\log|X_t|}{\pp}-\ab t\bigr) 
  - \bigl(Z_t-\ab t\bigr)\right| }{\sqrt{t}}\\
&= \lim_{t\to\infty}\frac{\left|\frac{\log| X_t|}{\pp}\right| -\ab t}
  {\sqrt{t}} \qquad \quad(\text{since} \;\Prob[\log| X_t|<0]\to 0)\\
&= \lim_{t\to\infty}\frac{\frac{\log|X_t|}{\pp}  -\ab t}{\sqrt{t}}
= \mathcal{N}\,.
\end{aligned}
$$
On the other hand
$$
\lim_{t\to\infty}\frac{\dist_{\Sol}(\Zf_t\,,\,\of) - \ab\,t}{\sqrt{t}}
\ge \lim_{t\to\infty}\frac{|Z_t|-\ab t}{\sqrt{t}}= \mathcal{N}
\quad \text{in law.}
$$
When $\ab < 0$, the result follows once more by exchanging the roles
of the $x$- and $y$-coordinates.

\smallskip

Now consider the case when $\ab=0$.  Combining Theorem \ref{thm:clt}(iii) with
Proposition \ref{pro:metric}(iv), we obtain in law
$$
\begin{aligned}
\lim_{t\to\infty}\frac{\dist_{\Sol}(\Zf_t\,,\,\of)}{\sqrt{t}}
&\le \bigl|\overline{\!\mathcal M} - \underline{\mathcal M}\bigr|
 +\min\Bigl\{ \bigl|\overline{\!\mathcal M}\bigr|
   + \bigl|\mathcal{N}-\underline{\mathcal M}\bigr|\,,\, 
   \bigl|\underline{\mathcal M}\bigr|
   + \bigl|\overline{\!\mathcal M}-\mathcal{N}\bigr|\Bigr\}\\
&= 2\bigl(\,\overline{\!\mathcal M}-\underline{\mathcal M}\,\bigr)
   -|\mathcal{N}|.
\end{aligned}
$$
This upper bound together with the fact that 
$\dist_{\Sol}(\Zf_t\,,\,\of) \to \infty$ almost surely yields that 
$$
\frac{\log \dist_{\Sol}(\Zf_t\,,\,\of)}{\sqrt{t}} \to 0
$$ 
in law (and in fact almost surely). We can combine this with 
Theorem \ref{thm:clt}(iii) and Proposition \ref{pro:metric}(i), and get
the required lower bound in the case $\ab=0$.
\end{proof}

Compare this with the analogous result of \cite{Be} for simple random walk
with drift on Diestel-Leader graphs.
We next observe the following (denoting expectation by $\Ex$).

\begin{lem}\label{lem:distincrement}
Let $U_n = \max \bigl\{\dist_{\Sol}(\Zf_n\,,\,\Zf_{n+t}) : 
0 \le t \le 1\bigr\}$. Then $\Ex(U_n) < \infty\,$, and
$$
\lim_{n \to \infty} \frac{1}{\log n}
\max \bigl\{\dist_{\Sol}(\Zf_n\,,\,\Zf_{n+t}) : 
0 \le t \le 1\bigr\} = 0 \quad\text{almost surely.}
$$
\end{lem}

\begin{proof} 
%Set $U_n = \max \bigl\{\dist_{\Sol}(\Zf_n\,,\,\Zf_{n+t}) : 
%0 \le t \le 1\bigr\}$. 
The random variables $U_n$, $n \ge 0$, are i.i.d.
Let
$$
X_*=\max \{|X_t| :0\le t\le  1 \}\,,
\quad Y_*=\max \{|Y_t| :0\le t\le 1 \}\,,
\quad Z_*=\max \{|Z_t| :0\le t\le 1 \}\,.
$$
Then by Proposition \ref{pro:metric}(iii), 
$$
U_0 \le c + \frac{2}{\pp}\log(1+|X_*|) + \frac{2}{\qq}\log(1+|Y_*|) + |Z_*|\,.
$$
Observe that by the Burkholder--Davis--Gundy inequality 
\cite[pag. 161]{RY}, we have for every $r>2$ that
$$
\Ex(X_*^r)  \le c_1 
\Ex\left( \left(\int_0^1 e^{2\pp Z_s}\,ds\right)^{r /2}\right)
\le c_1 
\int_0^1 \Ex\bigl(e^{\pp r Z_s}\bigr)\, ds<\infty\,,
$$
where $c_1 > 0$.
The same holds for $Y_*$. For $Z_*\,$, observe that by duality
$$
\begin{aligned} 
\Prob[Z_*>z]
&=\Prob[\max\{W_t-\ab t : 0\le t\le 1\} > z \;\text{ or }\; 
\max\{-W_t+\ab t : 0\le t\le 1\} >z]\\
&\le 4\Prob[W_1>z-|\ab|\,]\,.
\end{aligned}
$$
Thus
$$
\Ex\bigl(e^{r Z_*}\bigr)
=\int_0^\infty\Prob[e^{r Z_*}>u]\, du
\le 4 \, \int_0^\infty\Prob[e^{r (W_1+|\ab|)}>u] \, du 
= \Ex\bigl(e^{r (W_1+|\ab|)}\bigr)<\infty\,.
$$
We find that for all $r > 0$,
$$
\Ex\bigl( e^{r U_n} \bigr) = \Ex\bigl( e^{r U_0} \bigr)
\le e^{c r} \Ex\Bigl( (1+X_*)^{2r/\pp}\, (1+Y_*)^{2r/\qq} \,
e^{r Z_*}\Bigr) < \infty\,.
$$
By the law of large numbers, 
$e^{r U_n}/n \to 0$ almost surely, whence
$\limsup_{n \to \infty} U_n /\log n < 1/r$ almost surely,
for all  $r > 0$.
\end{proof}

Given the (left) action on $\Sol(\pp,\qq)$ on itself
by isometries and the group-invariance of our Laplacian (Lemma 
\ref{eq:gtranslate}), we get that along any time interval $[s\,,\,t]$, 
the increment of our Brownian motion $\Zf_t= (X_t,Y_t,Z_t)$ of 
\eqref{eq:stochint} satisfies 
\begin{equation}\label{eq:increment}
\Zf_s^{-1}\Zf_t = \Zf_{t-s}\quad\text{in law,}
\end{equation}
and for an arbirary number of time intervals which
do not overlap (i.e., they meet at most at the endpoints), the
associated increments are independent.
We now also get the rate of escape for our Brownian motion with drift.

\begin{cor}\label{cor:rate} For any value of $\ab$,
$$
\lim_{t \to \infty} 
\frac{\dist_{\Sol}(\Zf_t\,,\,\of)}{t} = |\ab| \quad\text{almost surely.}
$$ 
\end{cor}

\begin{proof}
In view of Lemma \ref{lem:distincrement} and the spatial homogeneity 
\eqref{eq:increment}, the subadditive 
ergodic theorem of {\sc Kingman}~\cite{Ki} implies that $\dist_{\Sol}(\Zf_t\,,\,\of)/t$ converges almost
surely to a constant. (Compare with {\sc Derriennic}~\cite{De}
for the case of discrete time.)
Theorem \ref{thm:clt-dist} implies that the limit is $\ab$ in probability,
whence also with probability $1$.
\end{proof}

\section{Convergence to the boundary at infinity, and the deviation from the
limit geodesic} \label{sect:boundary}

The natural geometric compactification of the hyperbolic plane, in the unit disk
model, is just the closed (Euclidean) disk. In the upper half plane model
$\Hb(\pp)$, the boundary at infinity $\vartheta \Hb(\pp)$ of the compactification
$\wh \Hb(\pp)$ is obtained by adding the bottom line $\vartheta^* \Hb(\pp)=\R$ and the point
at infinity, denoted here by $\varpi_{\pp}\,.$
In the logarithmic model,  convergence to the boundary is as follows:
we have that $(x,z) \to \xi \in \vartheta^* \Hb(\pp)$
when $z \to -\infty$ and $x \to \xi$, and $(x,z) \to \varpi_{\pp}$ if
$|x| + e^z \to \infty$.

Now $\Sol(\pp,\qq)$ embeds into $\Hb(\pp) \times \Hb(\qq)$ via \eqref{eq:horpro}.
Therefore the most natural geometric compactification $\wh \Sol(\pp,\qq)$ of
$\Sol(\pp,\qq)$ is its closure in the compact bidisk $\wh \Hb(\pp) \times \wh \Hb(\qq)$.
(``Bidisk'' because when we use the unit disk model of hyperbolic plane,
this is just the direct product of two closed unit disks.) We assemble a brief
description of convergence to the boundary in the next lemma; no proof is required.
We underline once more the analogy with Diestel-Leader graphs \cite{Wo-lamp} and
treebolic space \cite{BSSW2}. As pointed out in the Introduction, the 
boundary at infinity is topologically a filled number ``8'', that is, two closed
disks glued
together at a single point. This sheds some light on the observations made
by {\sc Lyons and Sullivan}~\cite{LySu}.

\begin{lem} The boundary at infinity $\vartheta \Sol(\pp,\qq)$ of $\Sol(\pp,\qq)$ is
% % $$
% % \vartheta \Sol(\pp,\qq)
% % = \underbrace{\vartheta^* \Hb(\pp)}_{\displaystyle \R} 
% % \times \{ \varpi_{\qq} \}
% % \;\,\cup\,\; \{ \varpi_{\pp} \} \times 
% % \underbrace{\vartheta^* \Hb(\qq)}_{\displaystyle \R}
% % \,\;\cup\,\;\{ (\varpi_{\pp},\varpi_{\qq})\}\,.
% % $$
$$
%\vartheta \Sol(\pp,\qq) = 
\Bigl(\underbrace{\vartheta^* \Hb(\pp)}_{\displaystyle \R} 
\times \{ \varpi_{\qq} \}\Bigr)
%\;\,
\cup
%\,\; 
\Bigl(\{ \varpi_{\pp} \} \times 
\underbrace{\vartheta^* \Hb(\qq)}_{\displaystyle \R}\Bigr)
%\,\;
\cup%\,\;
\Bigl(\Hb(\pp)\times\{ \varpi_{\qq} \} \Bigr)
%\,\;
\cup%\,\; 
\Bigl(\{ \varpi_{\pp} \} \times \Hb(\qq)\Bigr)
%\,\;
\cup\;\Bigl\{ (\varpi_{\pp},\varpi_{\qq})\Bigr\}\,.
$$
Convergence to the boundary is as follows.
In general, 
$$
\xf = (x,y,z) \to (\xi, \eta) \in \vartheta \Sol(\pp,\qq)\,,\quad
\text{if}\quad (x,z) \to \xi \;\text{ in }\;  \wh \Hb(\pp) \;\text{ and }\;
(y,-z) \to \eta \;\text{ in }\;  \wh \Hb(\qq).
$$ 
This means that
$$
\begin{aligned}
\xf = (x,y,z) &\to (\xi, \varpi_{\qq}) 
\in \vartheta^* \Hb(\pp) \times \{ \varpi_{\qq} \}\,,\quad\text{if}\quad
x \to \xi\;\text{ and }\;z \to -\infty\,,\\
\xf = (x,y,z) &\to (\varpi_{\pp},\eta)
\in \{ \varpi_{\pp} \} \times \vartheta^* \Hb(\qq)\,,\quad\text{if}\quad
y \to \eta \;\text{ and }\; z\to \infty\,,\\
\xf = (x,y,z) &\to \bigl((x_0,z_0), \varpi_{\qq}\bigr) 
\in \Hb(\pp) \times \{ \varpi_{\qq} \}\,,\quad\text{if}\quad
x \to x_0\,,\; z \to z_0\;\text{ and }\;|y| \to \infty\,,\\
\xf = (x,y,z) &\to \bigl(\varpi_{\pp},(y_0,-z_0)\bigr)
\in \{ \varpi_{\pp} \} \times \Hb(\qq)\,,\quad\text{if}\quad
y \to y_0\,,\; z \to -z_0 \;\text{ and }\; |x|\to \infty\,,\\
\xf = (x,y,z) &\to (\varpi_{\pp}, \varpi_{\qq})\,,\quad\text{if}\quad
|x| + e^z \to \infty\;\text{ and }\; |y| + e^{-z} \to \infty\;\,\text{ in }\;\R\,.
\end{aligned}
$$
\end{lem}

A \emph{geodesic ray} is a continuous mapping $\gab: [0\,,\,\infty) \to\Sol$
(or to any of our other spaces) such that 
$\dist\bigl(\gab(t)\,,\,\gab(s)\bigr) = |t-s|$ for all $s, t$.
Its starting point is $\gab(0)$.
For any $(x_0,z_0) \in \Hb(\pp)$ and $\xi \in \vartheta \Hb(\pp)$, there is
a unique geodesic ray $\bigl(x(t),z(t)\bigr)$ that starts at $(x_0, z_0)$ and
converges to $\xi$. In the case when $\xi = \varpi_{\pp}$ then this is
the upwards going vertical half-line $t \mapsto (x_0, z_0+t)$ in $\Hb(\pp)$.
%% starting at $(x_0,z_0)$.

For $\xf = (x_0,y_0,z_0) \in \Sol(\pp,\qq)$ and a boundary point
$(\varpi_{\pp},\eta) \in \{ \varpi_{\pp} \} \times \vartheta^* \Hb(\qq)$
of $\Sol(\pp,\qq)$, we can consider the (unique) upwards geodesic ray 
starting at $\xf$ given by $\gab_{\xf}^{\eta}(t) 
= \bigl(x_0, y(t), z(t)\bigr)$, where $\bigl(y(t), -z(t)\bigr)_{t\ge 0}$ 
is the geodesic ray from $(y_0,-z_0)$ to $\eta$ in $\Hb(\qq)$.

Analogously, for a boundary point 
$(\xi, \varpi_{\qq}) \in \vartheta^* \Hb(\pp) \times \{ \varpi_{\qq} \}$,
we have the (unique) downwards geodesic ray 
starting at $\xf$ given by $\gab_{\xf}^{\xi}(t) 
= \bigl(x(t), y_0, z(t)\bigr)$, where $\bigl(x(t), z(t)\bigr)_{t\ge 0}$ 
is the geodesic ray from $(x_0,z_0)$ to $\xi$ in $\Hb(\pp)$.
All those geodesics converge to their defining boundary points, as 
$t \to \infty$, and any two geodescis that converge to the same boundary point
are at bounded Hausdorff distance. This is true because it holds in the
hyperbolic plane. 

In the first of the above two cases, it will be most convenient to use
the initial point $\xf = (0,\eta,0)$, and omit the index $\xf$ in that
case. Thus, we can parametrise by $z \ge 0$ and get
$\gab^{\eta}(z)= (0,\eta,z)$. 
Analogously, in the second case, we use the standard initial point 
$\xf = (\xi,0,0)$ and get the corresponding geodesic ray
$\gab^{\xi}(z)= (\xi,0,-z)$, again parametrised by
$z \ge 0$. 
We call these the (upwards, resp. downwards) \emph{vertical} geodesic
rays. We remark that there is no geodesic ray in $\Sol$ 
from any starting point that converges to $(\varpi_{\pp},\varpi_{\qq})$.

Compare with \cite{EFW1}, \cite{EFW2} for further details on the 
geometry.

Let us return to our Brownian motion $\Zf_t= (X_t,Y_t,Z_t)$ of 
\eqref{eq:stochint}. 
 
\begin{pro}\label{pro:finlim}
\emph{(i)} If $\ab > 0$ then 
$$
\lim_{t\to \infty} Y_t = Y_{\infty} 
= \int_0^{\infty} e^{-\qq Z_s}\, dW^{(2)}_s \quad \text{almost
surely.}
$$
That is, $\Zf_t \to (\varpi_{\pp}\,,Y_{\infty}) \in \vartheta \Sol(\pp,\qq)$
almost surely in the topology of $\wh \Sol(\pp,\qq)$.\\[5pt]
\emph{(ii)} If  $\ab < 0$ then 
$$
\lim_{t\to \infty} X_t = X_{\infty} 
= \int_0^{\infty} e^{\pp Z_s}\, dW^{(1)}_s \quad \text{almost
surely.}
$$
That is, $\Zf_t \to (X_{\infty}\,,\varpi_{\qq})$
almost surely in the topology of $\wh \Sol(\pp,\qq)$.\\[5pt] 
In both cases \emph{(i)} and \emph{(ii)}, the respective limiting random variable is a.s. 
finite.
\\[5pt]
\emph{(iii}) If $\ab=0$ then 
$$
|X_t| + e^{Z_t} \to \infty \AND |Y_t| + e^{-Z_t} \to \infty \quad
\text{almost surely.}
$$ 
That is, $\Zf_t \to (\varpi_{\pp}\,,\varpi_{\qq})$
almost surely in the topology of $\wh \Sol(\pp,\qq)$.
\end{pro}

\begin{proof}
(i) and (ii) are immediate from the representation \eqref{eq:stochint} via 
\cite[Prop. 1.26]{RY}. 

%% %\bigskip
%% %\hrule
%% %
%% %\smallskip
%% %
%% %{\sl Non essendo esperto di integrali stocastici, la mia
%% %dimostrazione andrebbe attraverso la passeggiata (destra)
%% %ottenunta ai tempi interi. Comunque, anche qui si 
%% %dovrebbe applicare il criterio della radice, cio\`e, 
%% %$$
%% %\lim_t |\text{funzione integranda}|^{1/t} < 1\,.
%% %$$
%% %Ma poi non mi e' chiarissimo in che modo interviene il 
%% %$dW^{(i)}_s$ per garantire la convergenza dell'integrale.
%% %}
%% %
%% %\smallskip
%% %
%% %\hrule
%% %\bigskip

For (iii), consider $\Xf_t = (X_t, Z_t)$ as a process on the affine group 
$\Aff(\pp)$ of \eqref{eq:affmatrix}.

It also satisfies \eqref{eq:increment} and \eqref{lem:distincrement}.
We consider our process at discrete times:
\begin{equation}\label{eq:xfn}
\Xf_n = (X_n,Z_n)
 = \Xf_1 \cdot (\Xf_1^{-1}\Xf_2) \cdots (\Xf_{n-1}^{-1}\Xf_n)
\end{equation}
is a right random walk on $\Aff(\pp)$.
We can apply a result of {\sc Brofferio}~\cite{Bro}. In the notation of
\cite{Bro}, $A_1 = e^{\pp Z_1}$ and $B_1 = X_1$. 
Since the expectation of
$\log A_1$ is $0$,
and all moment conditions of \cite[Thm. 1]{Bro} are satisfied,
$\Xf_n \to \varpi_{\pp}$ almost surely in $\Hb(\pp)$, as $n \to \infty$ in $\Z$.
By Lemma \eqref{lem:distincrement}, also $\Xf_t \to \varpi_{\pp}$ almost surely,
as $t \to \infty$ in $\R$.

In the same way, $(Y_t,-Z_t) \to  \varpi_{\qq}$ almost surely,
as $t \to \infty$ in $\R$. Statement (iii) follows.
\end{proof}

Thus, when $\ab > 0$, we have the vertical \emph{limit geodesic} 
$\gab^{Y_{\infty}}$ to whose limit point our Brownian motion converges, and when
$\ab < 0$ we have to replace this by $\gab^{X_{\infty}}$. In order to simplify
notation, we just write ${\gab^{\infty}}$ for the respective limit
geodesic in each of those cases.

We now prove that
when $\ab \ne 0$, the convergence of $\Zf_t$ to its boundary limit is
very straight, in the sense that its deviation from ${\gab^{\infty}}$
is of the order of $\log t$. 

\begin{thm}\label{thm:log-dist}
If $\ab \ne 0$ then Brownian motion on $\Sol(\pp,\qq)$ satisfies
$$
\limsup_{t \to \infty} \frac{1}{\log t}\, 
\dist\bigl( \Zf_t\,,\, \gab^{\infty}\bigr) \le 
%\frac{\pp+ \qq}{2|\ab|}\quad\text{almost surely,}
\frac{2}{|\ab|}\quad\text{almost surely,}
$$
where 
$\dist\bigl( \Zf_t\,,\, \gab^{\infty}\bigr) = 
\inf \bigl\{\dist\bigl( \Zf_t\,,\, \gab^{\infty}(u)\bigr) : u \ge 0 \bigr\}$.
\end{thm}

\begin{proof}
Once more, it is sufficient to consider only the case $\ab > 0$.

For each $t \ge 0$, the point $(0,Y_{\infty}, Z_t)$ lies on the
geodesic $\gab^{\infty} = \gab^{Y_{\infty}}$.
We shall show that for integer $n$,
\begin{equation}\label{eq:integ}
\limsup_{n \to \infty} \frac{1}{\log n}\, 
\dist\bigl( \Zf_n\,,\, (0,Y_{\infty}, Z_n)\bigr) \le 
\frac{2}{\ab}\quad\text{almost surely,}
\end{equation}
Together with Lemma \ref{lem:distincrement}, this will yield the result.

The metric $\dist_{\Sol}$ is invariant under the left action of
the group $\Sol(\pp,\qq)$. Using the product formula \eqref{eq:product} and
subsequently Proposition \ref{pro:metric}(iii), we find
$$
\begin{aligned}
\dist_{\Sol}\bigl(\Zf_t\,,(0,Y_\infty,Z_t)\bigr)
&=\dist_{\Sol}\bigl((e^{-\pp Z_t} X_t\,,
e^{\qq Z_t}(Y_t-Y_{\infty}),0),\of\bigr)\\
&\le
c  + \frac{2}{\pp}\log\bigl(1+e^{-\pp Z_t}|X_t|\bigr) 
+ \frac{2}{\qq}\log\bigl(1+e^{\qq Z_t}|Y_t-Y_{\infty}|\bigr) \,.
\end{aligned}
$$
We have
$$   %\begin{equation}\label{eq:law}
\begin{aligned}
e^{\qq Z_t}(Y_{\infty}-Y_t)
&=e^{\qq (W_t+at)}\int_t^{\infty} e^{-\qq (W_s+\ab s)}\, dW^{(2)}_s
=\int_t^{\infty} e^{-\qq ((W_s-W_t)+\ab (s- t))}\, dW^{(2)}_s\\
&\hspace{-.28cm}\stackrel{\mathrm{in\ law}}{=}
\int_0^\infty e^{-\qq (W_{s}+\ab s)}\, dW^{(2)}_{s} = Y_{\infty}\,.
\end{aligned}
$$  %\end{equation}
Recall from the proof of Proposition \ref{pro:finlim} that $(Y_n,-Z_n)$ can be
interpreted as the right random walk 
$\begin{pmatrix} e^{-\qq Z_n} & Y_n \\ 0 & 1\end{pmatrix}$ 
on the affine group. 
We can apply a theorem of {\sc Kesten}~\cite[Thm. B]{Kes} to the sequence
$(Y_n)$ and its limit $Y_{\infty}\,$. Namely, if we set 
$\kappa(\qq) = 2\ab/\qq$ then 
$\Ex\bigl((e^{-\qq Z_1})^{\kappa(\qq)}\bigr) =1$, whence
$$
\Prob[|Y_{\infty}| > y] \asymp y^{-\kappa(\qq)}\quad\text{as }\;y \to \infty\,.
$$ 
Now take $\delta>1/\kappa$. Then
$$
\begin{aligned}
\sum_{n=2}^{\infty}
\Prob\Bigl[\log\bigl(1+e^{\qq Z_n}|Y_n-Y_{\infty}|\bigr)&> \delta \log n\Bigr]
= \sum_{n=2}^\infty\Prob\Bigl[\log\bigl(1+|Y_{\infty}|\bigr)> \delta \log n\Bigr]\\
&= \sum_{n=2}^\infty \Prob\bigl[|Y_{\infty}|>n^\delta -1]  
\asymp \sum_{n=2}^\infty (n^\delta-1)^{-\kappa(\qq)}<+\infty\,.
\end{aligned}
$$
Thus by the Borel--Cantelli Lemma
$$
\limsup_{n\to\infty}\frac{1}{\log n} 
\log\bigl(1+e^{\qq Z_n}|Y_n-Y_\infty|\bigr)\le \frac{\qq}{2\ab}\quad\text{almost surely.}
$$
We now consider the first coordinate.  For fixed $t$ observe that
$$
\begin{aligned}
e^{-\pp Z_t} X_t&=e^{-\pp (W_t+\ab t)}\int_0^t e^{\pp (W_s+\ab s)}\, dW^{(1)}_s\\
&=\int_0^t e^{-\pp ((W_t-W_s)+\ab (t- s))}\, dW^{(1)}_s\\
\hspace{-.28cm}&\stackrel{\mathrm{in\ law}}{=}
\int_0^t e^{-\pp (W_s +\ab s)}\, dW^{(1)}_s=: \wt X_t\,,
\end{aligned}
$$
and $\wt X_{\infty}= \lim_{t \to \infty} \wt X_t$ exists almost surely.
%Re-using the argument of \eqref{eq:law}
As above, one finds that 
$$
e^{\pp Z_t}(\wt X_{\infty} - \wt X_t)  
=\int_t^{\infty} e^{-\pp ((W_s-W_t)+\ab (s- t))}\, dW^{(1)}_s =:
\overline {\!X}_{t,\infty}\,,
$$ 
where $\overline {\!X}_{t,\infty}$ is independent from $(\Zf_s)_{0 \le s \le t}$
and has the same law as $\wt X_{\infty}$. Thus, for some constant $C > 0$ and
for any $x > 0$, 
$$
\begin{aligned}
\Prob\bigl[e^{-\pp Z_t} |X_t|>x\bigr] 
&= \Prob\bigl[|\wt X_n |>x\,,\; |\wt X_{\infty}-\wt X_n| \le \tfrac{x}{2}\bigr]
+ \Prob\bigl[|\wt X_n |>x\,,\; |\wt X_{\infty}-\wt X_n| > \tfrac{x}{2}\bigr]\\
&\le \Prob\bigl[|\wt X_{\infty}|>\tfrac{x}{2}\bigr]
  + \Prob\bigl[ e^{-\pp Z_n}|\,\overline{\!X}_{n,\infty}|> \tfrac{x}{2}\bigr]\\
&\le C \,\bigl(\tfrac{x}{2}\bigr)^{\kappa(\pp)} 
+ \Ex\Bigl( \Prob\bigl[ |\,\overline{\!X}_{n,\infty}|> \tfrac{x}{2}e^{\pp Z_n}
\,\big|\, Z_n\bigr]\Bigr)\\
&\le C \,\bigl(\tfrac{x}{2}\bigr)^{-\kappa(\pp)} 
+ C\,\Ex\Bigl( \bigl(\tfrac{x}{2}e^{\pp Z_n}\bigr)^{-\kappa(\pp)}\Bigr)\\
&\le C\, \bigl(\tfrac{x}{2}\bigr)^{-\kappa(\pp)} 
+ C\bigl(\tfrac{x}{2}\bigr)^{-\kappa(\pp)} 
\bigl.\underbrace{\Ex\bigl(e^{-\kappa(\pp)\pp Z_1}\bigr)}_{\displaystyle = 1}
\bigr.^{n}
= 2C \bigl(\tfrac{x}{2}\bigr)^{-\kappa(\pp)}
\end{aligned}
$$
Proceeding as above, the Borel--Cantelli Lemma implies that
$$
\limsup_{n\to\infty}\frac{1}{\log n} \log\bigl(1+ e^{-\pp Z_n} |X_n|\bigr)
\le \frac{\pp}{2\ab}
$$
almost surely.
\end{proof}

\section{Elements of potential theory}\label{sect:potential}

If $\Lap$ is any of our different Laplacians on $\Sol$, $\Hb$, or $\R$,
and $\la \in \R$, then we denote by $\Hc(\Lap, \la)$ the space of all
functions $h$ on our space which satisfy $\Lap h = \la\cdot h$. The positive
cone $\Hc^+(\Lap,\la)$ contains non-zero functions if and only if
$\la \ge \la_{\min}(\Lap)$, the \emph{bottom of the positive spectrum.}
Below we shall clarify what the values of $\la_{\min}$ are in each of
our cases. In any case, $\la_{\min} \le 0$, since the space of
\emph{harmonic functions} $\Hc(\Lap) = \Hc(\Lap, 0)$ contains all
constant functions.

By the \emph{minimum principle,} every non-zero function in $\Hc^+(\Lap,\la)$
must be strictly positive in each point.

A function $h$ in $\Hc^+(\Lap,\la)$ is called \emph{minimal} if $h(0)=1$
and whenever $h \ge f \in \Hc^+(\Lap,\la)$ then $f/h$ is constant.
A basic fact in classical potential theory of Riemannian manifolds
says that every function in $\Hc^+(\Lap,\la)$ can be expressed uniquely
as an integral over the minimal harmonic functions with respect to a
finite Borel measure on the latter set. 
%% %(``Borel measure'' refers to the
%% %topology of uniform convergence on compact sets.) 
%% %\bigskip
%% %\hrule
%% %\smallskip
%% %{\sl \`E la topologia giusta?}
%% %\smallskip
%% %\hrule
%% %\bigskip

We shall specify this in more detail
in our cases below.

Let us now recall what happens in the case of the standard Laplacian
$$
\Lap^{\Hb}_{-1/2} = \frac12\Bigl(e^{2 z}\,\frac{\bd^2}{\bd x^2} +
 \frac{\bd^2}{\bd z^2}  - \frac{\bd}{\bd z}\Bigr)
$$
on standard hyperbolic plane $\Hb = \Hb(1) = \{ x + \im e^z : x,z \in \R\}$ in
the logarithmic model. 
% % %The (hyperbolic) boundary $\vartheta \Hb = \R \cup \{ \varpi \}$ 
% % %is the image of the unit circle,
% % %which is the boundary in the unit disc model of $\Hb$, under the standard
% % %M\"obius mapping that maps $0$ to $\im$, $\im$ to $\infty$ and the rest
% % %of the unit circle to the real axis. (In particular, topologically
% % %$\vartheta\Hb$ is obtained from $[-\infty\,,\,+\infty]$ by identifying
% % %$+\infty$ and $-\infty$.)

The minimal harmonic functions are the \emph{Poisson kernels,} which are 
parametrised by the (hyperbolic) boundary 
$\vartheta \Hb = \R \cup \{ \varpi \}$. (Recall that $\varpi=\varpi_1$ is the
point at infinity.) In the logarithmic model the kernels are
\begin{equation}\label{eq:Poisson1}
P\bigl( (x,z), \varpi \bigr) = e^z \AND
P\bigl( (x,z), \xi \bigr) = \frac{(\xi^2+1)e^z}{(\xi-x)^2 + e^{2z}}\,,\quad
\xi \in \R\,.
\end{equation}
We have $\la_{\min}(\Lap^{\Hb}_{-1/2}) = -1/8$, and the minimal elements in
$\Hc^+(\Lap^{\Hb}_{-1/2},\la)$ are the functions
\begin{equation}\label{eq:Poisson2}
P(\cdot,\xi)^{\al(\la)}\,,\quad\text{where}\quad \xi \in \vartheta\Hb
\AND \al(\la)= \frac{1+\sqrt{1+8\la}}{2}\,.
\end{equation}

Next, let us turn our attention to $\Hc^+(\Lap^{\Hb(\pp)}_{\ab},\la)$.

\begin{lem}\label{lem:transfer}
A function $f$ on $\/\Hb(\pp)$ is in
$\Hc\bigl(\Lap^{\Hb(\pp)}_{\ab},\la\bigr)$
if and only if the function on $\/\Hb(1)$ given by
$\bigl(e^{(\ab+\pp/2)z}f\bigr)\!\circ\theta$ is in
$\Hc\bigl(\Lap^{\Hb(1)}_{-1/2},\tfrac{8\la + 4\ab^2 -\pp^2}{8\pp^2}\bigr)\,$,
where $\theta(x,z) = (x/\pp, z/\pp)$.
In particular,
$$
\la_{\min}(\Lap^{\Hb(\pp)}_{\ab}) = -\ab^2/2\,,
$$
and for $\la \ge -\ab^2/2$, the minimal elements in
$\Hc^+\bigl(\Lap^{\Hb(\pp)}_{\ab},\la\bigr)$ are the functions
$$
P_{\pp,\ab,\la}\bigl((x,z),\varpi_{\pp}\bigr) = e^{\alpha z}
\!\!\AND\!\! P_{\pp,\ab,\la}\bigl((x,z),\xi\bigr) = e^{\alpha z}
\left(\frac{(\xi^2+1)}{(\xi-\pp\,x)^2 + e^{2\pp z}}
\right)^{\!\beta}\!\!\!, \quad \xi \in \vartheta^*\Hb(\pp) \equiv\R\,,
$$
where $\alpha = \alpha(\la,\ab)= \sqrt{\ab^2+2\la}-\ab$ and 
$\beta = \beta(\la, \ab,\pp) = \frac{1}{2} + \frac{\sqrt{\ab^2+2\la}}{\pp}\,$.
\end{lem}

\begin{proof} First of all, it is a straightforward computation that
$$
\bigl(\Lap^{\Hb(\pp)}_{\ab}f\bigr) \! \circ \theta =
\pp^2 \, \Lap^{\Hb(1)}_{\ab/\pp}(f \! \circ \theta)\,.
$$
Therefore $f$ is in $\Hc(\Lap^{\Hb(\pp)}_{\ab},\la)$ if  and only if
$\bar f = f \!\circ \theta$ is in $\Hc(\Lap^{\Hb(1)}_{\ab/\pp},\la/\pp^2)\,$.

For the moment, set $\bar \ab = \ab/\pp$ and
$\bar\la = \la/\pp^2$. Then we compute
$$
\Lap^{\Hb(1)}_{-1/2}\bigl( e^{(\bar\ab+1/2)z}\bar f \bigr)
= e^{(\bar\ab+1/2)z} \Bigl(\Lap^{\Hb(1)}_{\bar\ab}\bar f
		       + \tfrac{4\bar\ab^2-1}{8}\bar f\Bigr)\,.
$$
Therefore $\bar f$ is in $\Hc(\Lap^{\Hb(1)}_{\bar\ab},\bar\la)$ if and only if
$e^{(\bar\ab+1/2)z}\bar f$ is in
$\Hc(\Lap^{\Hb(1)}_{-1/2},\bar\la+\tfrac{4\bar\ab^2-1}{8})$.

Combining these computations, the statements follow.
\end{proof}

Thus, every function $h \in \Hc^+(\Lap^{\Hb(\pp)}_{\ab},\la)$ has a unique
integral representation
\begin{equation}\label{eq:Poisson3}
h(x,z) = \int_{\vartheta\Hb}
P_{\pp,\ab,\la}\bigl( (x,z), \xi \bigr)\, d\nu(\xi)\,,
\end{equation}
where $\nu$ is a (finite, positive) Borel measure on $\vartheta\Hb\,$. (This
includes $\xi = \varpi_{\pp}\,$.)

\begin{rmk}\label{rmk:Liouville}
When $\la =0$, that is, when we consider ordinary harmonic functions,
we see that the constant function $\uno$ is minimal harmonic if and only if 
$\ab \ge 0$. This can be stated also by saying that 
$\Lap^{\Hb(\pp)}_{\ab}$ has the (weak) Liouville property, i.e., all bounded
harmonic functions are constant, precisely when $\ab \ge 0$.
\end{rmk}

%\smallskip

Everything that we have said so far in this section is very well known; see
e.g. {\sc Helgason} \cite{He}, or many other sources.

\smallskip

Let us now turn our attention to $\Sol$. The following is immediate.
% from Lemma \ref{lem:lift}

\begin{lem}\label{lem:lift-minimal}
If the function $h_1$ on $\Hb(\pp)$ is such that $h = h_1 \circ \pi_1$
is minimal in $\Hc^+\bigl(\Lap^{\Sol(\pp,\qq)}_{\ab},\la\bigr)$, then
$h_1$ is also minimal in $\Hc^+\bigl(\Lap^{\Hb(\pp)}_{\ab},\la\bigr)$.

In the same way, if the function $h_2$ on $\Hb(\qq)$ is such that
$h = h_2 \circ \pi_2$ is minimal in
$\Hc^+\bigl(\Lap^{\Sol(\pp,\qq)}_{\ab},\la\bigr)$, then
$h_2$ is also minimal in $\Hc^+\bigl(\Lap^{\Hb(\qq)}_{-\ab},\la\bigr)$.
\end{lem}

We now need a part of the Martin boundary theory for elliptic operators
on manifolds. The reader is referred to {\sc Ancona}~\cite{An} and
{\sc Taylor}~\cite{Ta} for the necessary backgroud material.
See also \cite[Chapter VI]{GJT}.
In the following propositions, we subsume the necessary material without all proofs.

\begin{pro}\label{pro:heatkernel}
The Markov semigroup $\Hf_t = \Hf_t^{\ab} = \exp(t\Lap_{\ab})$, $t > 0$,
admits a symmetric, bounded kernel $\hb_t(\xf,\zf) = \hb_t^{\ab}(\xf,\zf)$
with respect to the measure $\ms_{\ab}$ of Lemma \ref{lem:reversible}(a),
such that
$$
\Hf_t f(\xf) = \int_{\Sol} \hb_t(\xf,\zf)f(\zf)\,d\ms_{\ab}(\zf)\,.
$$
For each $\zf = (x,y,z) \in \Sol$, the function $\hb_t(\cdot,\zf)$ is
in $C^2(\Sol)$.
Furthermore, its kernel with respect to the volume element $d\zf$
of the $\Sol$-manifold,
$$
\pb_t(\xf,\zf) = \hb_t(\xf,\zf)\,e^{(\ab + \pp - \qq)z}\,,
$$
is stochastic and invariant under the action of the group $\Sol(\pp,\qq)$.
\end{pro}

\begin{pro}\label{pro:greenkernel}
The associated Green kernel
$$
\gb_{\ab}(\xf,\zf|\la) = \gb(\xf,\zf|\la)
=\int_0^\infty e^{-\la t}\pb_t(\xf,\zf)\,dt \quad
(\xf, \zf \in \Sol\,,\;\xf \ne\zf)
$$
is strictly positive and finite for each $\la \ge \la_{\min}(\Lap_{\ab})$.
%and for each $\zf \in \HT$, the function $\gb_t(\cdot,\zf)$ is
%in $\CC^{(2,\beta)}(\HT\setminus \{\zf\})$.
\end{pro}

We remark that finiteness at $\la = \la_{\min}(\Lap_{\ab})$ follows
from the fact that the cone of positive eigenfunctions
$\Hc^+\bigl(\Lap^{\Sol(\pp,\qq)}_{\ab},\la_{\min}\bigr)$ does not
collapse to a single half-line, as one can see from Lemma \ref{lem:transfer}
combined with \eqref{eq:Poisson1} and \eqref{eq:Poisson2}.

\begin{pro}\label{pro:harnack}
For each $d > 0$ and $\la \ge \la_{\min}$, the Green kernel satisfies
the Harnack inequality
$$
\frac{\gb(\xf,\zf'|\la)}{\gb(\xf,\zf|\la)} \le C_d(\la) \AND
\frac{\gb(\zf',\xf|\la)}{\gb(\zf,\xf|\la)} \le C_d(\la) \,,
$$
whenever $\dist(\zf,\zf') \le d$ and
$\min \{ \dist(\zf,\xf), \dist(\zf',\xf) \} \ge 10 (d+1)$, where
$C_d(\la) > 1$ is such that $C_d(\la) \to 1$ when $d \to 0$.

Furthermore, every function $h$ in
$\Hc^+\bigl(\Lap^{\Sol(\pp,\qq)}_{\ab},\la\bigr)$ satisfies
$$
\frac{h(\zf')}{h(\zf)} \le C_d(\la) \quad \text{for all}\quad \zf, \zf' \in
\Sol \;\text{ with }\;\dist(\zf,\zf') \le d.
$$
\end{pro}

\begin{proof}[Proof (outline)] 
In the case $\Lap_{\ab}$ is the Laplace-Beltrami operator of $\Sol(\pp,\qq)\,$, 
one can apply well-known Harnack inequalities of {\sc Li and Yau}, see \cite{LY},
because the Riemannian structure is invariant under a group action and thus 
the Ricci curvature is bounded below. 

For arbitrary values of $\ab$, our operator is obtained by adding to the  
Laplace-Beltrami operator a multiple of $\frac{\partial}{\partial z}$, which 
leads just to conjugating
our functions with an exponential in $z$, compare with the proof of 
Lemma \ref{lem:transfer}. Thus, the inequalities hold with any drift
term $\ab$.
\end{proof}

The \emph{Martin kernel} is
\begin{equation}\label{eq:Martin}
\kb_{\ab}(\xf,\zf|\la) = \kb(\xf,\zf|\la)
= \frac{\gb_{\ab}(\xf,\zf|\la)}{\gb_{\ab}(0,\zf|\la)}\,,\quad \zf \ne 0, \xf\,.
\end{equation}

The \emph{Martin compactification} is the smallest compactification 
of the underlying space $\Sol$ (i.e., a Hausdorff space into which
$\Sol$ embeds homeomorphically and densely) such that each function
$\kb_{\ab}(\xf,\cdot|\la)$ has a continuous extension in the second
variable. The \emph{Martin boundary} $\Mart(\la) = \Mart(\Lap_{\ab},\la)$ 
is the ideal boundary added to the space in that compactification. 
The extended kernel is also denoted $\kb_{\ab}(\cdot,\cdot|\la)$
%The last fact that we shall need is the following.

\begin{pro}\label{pro:limit}
Every minimal eigenfunction $h$ in
$\Hc^+\bigl(\Lap^{\Sol(\pp,\qq)}_{\ab},\la\bigr)$, $\la \ge \la_{\min}$,
is of the form
$$
h(\xf) = \kb_{\ab}(\xf,\zeta|\la), \quad\text{where }\; \zeta \in \Mart(\la)\,.
$$
That is, there is a (suitable) sequence $(\zf_n)$ in $\Sol$ with
$\dist(0,\zf_n) \to \infty$, such that
$$
h(\xf) = \lim_{n \to \infty} \kb(\xf,\zf_n|\la)\,.
$$
\end{pro}

The \emph{minimal Martin boundary} $\Mart_{\min}(\la) = \Mart_{\min}(\Lap_{\ab},\la)$
consists of all $\zeta \in \Mart(\la)$ for which $\kb_{\ab}(\cdot,\zeta|\la)$
is minimal. It is a Borel subset of $\Mart(\la)$.
The \emph{Poisson-Martin representation theorem} says the following.

\begin{pro}\label{pro:martin}
For every function $h \in \Hc^+\bigl(\Lap^{\Sol(\pp,\qq)}_{\ab},\la\bigr)$,
there is a unique Borel measure $\nu^h$ on $\Mart_{\min}(\la)$
such that 
$$
h(\xf) = \int_{\Mart_{\min}(\la)} \kb(\xf, \cdot|\la) \,d\nu^h\quad 
\text{for every } \xf \in \Sol.
$$
\end{pro}

All this is of course true for more general manifolds and elliptic
operators; see \cite{Ta}.

While we are not able to determine the whole Martin compactification, that is,
the directions of convergence of the Martin kernels, we shall determine
precisely the minimal positive $\la$-eigenfunctions for each $\la \ge \la_{\min}$.

\section{Positive harmonic functions on $\Sol(\pp,\qq)$}\label{sect:main}

We now show that every positive eigenfunction of our Laplacian on $\Sol(\pp,\qq)$ 
splits as a sum of two eigenfunctions that live on the two respective hyperbolic
planes which make up $\Sol$, and we determine precisely all minimal positive
eigenfunctions. The first step is the following.

\begin{thm}\label{thm:minimal}
Let $h \in \Hc^+\bigl(\Lap^{\Sol(\pp,\qq)}_{\ab},\la\bigr)$ be minimal, where 
$\la \ge -\ab^2/2$. 

Then $h(x,y,z) = h_1(x,z)$, where $h_1$ is minimal in 
$\Hc^+\bigl(\Lap^{\Sol(\pp,\qq)}_{\ab},\la\bigr)$, or
$h(x,y,z) = h_2(y,-z)$, where $h_2$ is minimal in 
$\Hc^+\bigl(\Lap^{\Hb(\qq)}_{-\ab},\la\bigr)$.
\end{thm}

\begin{proof}
Let $h$ be a minimal eigenfunction in $\Hc^+\bigl(\Lap_{\ab},\la\bigr)$.
Then
$h = \lim_{n \to \infty} \kb(\cdot,\zf_n|\la)$.
Write $\zf_n = (x_n,y_n,z_n)$.
\\[4pt]
\emph{Claim.} (a) If $\inf_n z_n > -\infty\,$, then for each
$a \in \R$,
$$
h(x+a,y,z) = h(x,y,z) \quad \text{for all}\quad
(x,y,z) \in \Sol\,.
$$
(b) If $\sup_n z_n < +\infty\,$, then for each
$b \in \R$,
$$
h(x,y+b,z) = h(x,y,z) \quad \text{for all}\quad
(x,y,z) \in \Sol\,.
$$
To prove part (a) of this claim, let $a \in \R$ and consider the group element
$$
\gf_a = (a,0,0) = \left(\begin{smallmatrix} 1 & a & 0         \\[2pt]
				    0 & 1 & 0         \\[2pt]
				    0 & 0 & 1
		  \end{smallmatrix}\right) \in \Sol(\pp,\qq)\,.
$$
We abbreviate $\tau_a = \tau_{\gf_a}$. Let $\xf = (x,y,z) \in \Sol$.
Then, by \eqref{eq:gtranslate},
$\tau_a h(\xf) = h(\gf_a \xf) = h(x+a,y,z)$, and Lemma
\ref{lem:lap-inv} tells us that $\tau_a h$ is in
$\Hc^+\bigl(\Lap_{\ab},\la\bigr)$.
Now by \eqref{eq:preservedist}
$$
\dist(\gf_a\zf_n\,,\zf_n) = \dist_{\Sol}\bigl((x_n+a,y_n,z_n),(x_n,y_n,z_n)
= \dist_{\Hb(\pp)}\bigl((x_n+a,z_n),(x_n,z_n)\bigr)\,.
$$
Elementary properties of the hyperbolic metric imply that
$$
\dist_{\Hb(\pp)}\bigl((x_n+a,z_n),(x_n,z_n)\bigr) =
\dist_{\Hb(\pp)}\bigl((a,z_n),(0,z_n)\bigr)
\le \dist_{\Hb(\pp)}\bigl((a,c),(0,c)\bigr) = d_a\,,
$$
where $c= \inf_n z_n$.
Let $C_{d_a}(\la)$ be the corresponding Harnack constant in Lemma
\ref{pro:harnack}.
Then, using that $\gf(\cdot,\cdot|\la)$ is $\Sol(\pp,\qq)$-invariant,
$$
\kb(\gf_{a}\xf\,,\zf_n|\la) =
\frac{\gb(\gf_{a}\xf\,,\zf_n|\la)}{\gb(\gf_{a}\xf\,,\gf_{a}\zf_n|\la)}\,
\frac{\gb(\gf_{a}\xf\,,\gf_{a}\zf_n|\la)}{\gb(0\,,\zf_n|\la)}
\le C_{d_a} \, \kb(\xf\,,\zf_n|\la)\,.
$$
Letting $n \to \infty$, we obtain
$$
\tau_{a}h(\xf)= h(\gf_{a}\xf) \le C_{d_a} \, h(\xf) \quad
\text{for all} \; \xf \in \Sol\,.
$$
Now minimality of $h$ implies that the function $\tau_{a}h/h$ is constant.
For $\xf = (x,y,z) \in \Sol$,
$$
\dist_{\Sol}(\gf_a \xf,\xf) = \dist_{\Hb(\pp)}\bigl((x+a,z),(x,a)\bigl) \to 0\,,
\quad\text{if}\quad z \to +\infty\,.
$$
Therefore the second statement in Lemma \ref{pro:harnack} implies
that $h(\gf_{a}\xf)/h(\xf) \to 1$ as $z \to +\infty\,$, and we conclude
that $\tau_{a}h/h \equiv 1\,$.
This proves statement (a) of the claim, and statement (b) follows by exchanging
the roles of the $x$- and $y$-coordinates and changing the sign of~$z$.

\smallskip

Now $(z_n)$ must have a subsequence which converges to a limit
in $[-\infty\,,\,+\infty]$. We may assume without loss of generality
that $(z_n)$ itself converges.
\\[4pt]
\emph{Case 1.} $\; z_n \to \infty\,.$ Then we can apply part (a) of the
Claim, and conclude that $h$ depends only on $(y,z)$. By Lemma
\ref{lem:lift-minimal}, there is a function $h_2$ on $\Hb(\qq)$
which is minimal in $\Hc^+\bigl(\Lap^{\Hb(\qq)}_{-\ab},\la\bigr)$,
such that $h(x,y,z) = h_2(y,-z)$ for all $\xf = (x,y,z) \in \Sol$.
\\[4pt]
\emph{Case 2.} $\; z_n \to -\infty\,.$ Then we can apply part (b) of the
Claim, and again by Lemma
\ref{lem:lift-minimal}, there is a function $h_1$ on $\Hb(\pp)$
which is minimal in $\Hc^+\bigl(\Lap^{\Hb(\pp)}_{\ab},\la\bigr)$,
such that $h(x,y,z) = h_1(x,z)$ for all $\xf = (x,y,z) \in \Sol$.
\\[4pt]
\emph{Case 3.} $\; z_n \to z_0 \in \R$. Then we can apply both parts (a) and
(b) of the claim, and there is a function $\wt h$ on $\R$
such that $h(x,y,z) = \wt h(z)$ for all $\xf = (x,y,z) \in \Sol$.
It must be minimal both as a function $(x,z) \mapsto \wt h(z)$
in $\Hc^+\bigl(\Lap^{\Hb(\pp)}_{\ab},\la\bigr)$ and as a function
$(y,z) \mapsto \wt h(-z)$ in $\Hc^+\bigl(\Lap^{\Hb(\qq)}_{-\ab},\la\bigr)$.
Of course, it must also be a minimal element of
$\Hc^+\bigl(\Lap^{\R}_{\ab},\la\bigr)$.
\end{proof}

\begin{rmk}\label{rem:exponential}
If $h \in \Hc^+\bigl(\Lap^{\Sol(\pp,\qq)}_{\ab},\la\bigr)$ is minimal
and depends only on $z$ then it must arise by lifting a minimal element of 
$h \in \Hc^+\bigl(\wt \Lap_{\ab},\la\bigr)$ from $\R$ to $\Sol$. 
That is, we must have 
$h(x,y,z) = e^{\alpha z}$, where $\alpha =\pm \sqrt{\ab^2+2\la}-\ab$.
Furthermore, in this case, the function $(x,z) \mapsto e^{\alpha z}$
must be minimal in $\Hc^+\bigl(\Lap^{\Hb(\pp)}_{\ab},\la\bigr)$,
so that -- by Lemma \ref{lem:transfer} -- we can only have the ``+'' sign,
that is, $\alpha =\alpha(\la,\ab)$.
We shall see below that the corresponding function can really be a minimal 
$\la$-eigenfunction  on $\Sol$ only when $\la = \la_{\min}\,$.
\end{rmk}

\begin{cor}\label{cor:harmonic-sum} 
If $h \in \Hc^+\bigl(\Lap^{\Sol(\pp,\qq)}_{\ab},\la\bigr)$, where 
$\la \ge -\ab^2/2$, then there are nonnegative functions
$h_1 \in \Hc^+\bigl(\Lap^{\Hb(\pp)}_{\ab},\la\bigr)$ and
$h_2 \in \Hc^+\bigl(\Lap^{\Hb(\qq)}_{-\ab},\la\bigr)$ such that
for all $\xf = (x,y,z) \in \Sol(\pp,\qq)$,
$$
h(x,y,z) = h_1(x,z) + h_2(y,-z).
$$
\end{cor}

\begin{proof}
We see from Theorem \ref{thm:minimal} that the set of all minimal 
$\la$-eigenfunctions on $\Sol(\pp,\qq)$
is contained in the union of the sets of minimal $\la$-eigenfunctions on 
$\Hb(\pp)$ and $\Hb(\qq)$, with a change of the sign of $z$ for the latter, 
according to the above cases.
Thus, taking into account Remark \ref{rem:exponential}, 
$\Mart_{\min}(\la)$ can be parametrised by a subset of
the disjoint union 
$\vartheta \Hb(\pp) \cup \vartheta^* \Hb(\qq) \cong 
\Bigl(\vartheta \Hb(\pp) \times \{ \varpi_{\qq} \}\Bigr) \cup 
\Bigl(\{ \varpi_{\pp} \} \times \vartheta^* \Hb(\qq)\Bigr) \subset \Sol(\pp,\qq)$,
or in other terms, of the ``8''-shaped outer part of the geometric boundary of $\Sol$ 
(without the interiors of the two disks).

By Proposition \ref{pro:martin}, for every function 
$h \in \Hc^+\bigl(\Lap^{\Sol(\pp,\qq)}_{\ab},\la\bigr)$, there is a Borel
measure $\nu=\nu^h$ on $\Mart_{\min}(\la)$
that yields the integral representation
of $h$. Now let $\nu_1$ be the restriction of $\nu$ to $\vartheta \Hb(\pp)$
and $\nu_2$ the restriction to  $\vartheta^* \Hb(\qq)$. 
Then we get for every $\xf = (x,y,z) \in \Sol$ 
$$
h(\xf) = \int_{\vartheta \Hb(\pp)}\!\! 
P_{\pp,\ab,\la}\bigl((x,z),\xi\bigr)\, d\nu_1(\xi)
+ \int_{\vartheta^* \Hb(\qq)}\!\! 
P_{\qq,-\ab,\la}\bigl((y,-z),\eta\bigr)\, d\nu_2(\eta)
= h_1(x,z) + h_2(y,-z)\,,
$$
as proposed.
\end{proof}

\begin{cor}\label{cor:Liouville} The Laplacian $\Lap^{\Sol(\pp,\qq)}_{\ab}$
has the (weak) Liouville property, i.e., all bounded harmonic functions
on $\Sol$ are constant, if and only if the rate of escape
$\ab$ vanishes.
\end{cor} 

\begin{proof} If $\ab =0$, then all bounded harmonic functions are constant by
Corollary \ref{cor:harmonic-sum} and Remark \ref{rmk:Liouville}.
Conversely, if $\ab \ne 0$, then again by
Remark \ref{rmk:Liouville}, one of  $\Lap^{\Hb(\pp)}_{\ab}$ and 
$\Lap^{\Hb(\qq)}_{-\ab}$ has non-constant bounded 
harmonic functions, and they lift to harmonic functions on $\Sol(\pp,\qq)$.
\end{proof}

The last corollary, which was obtained in a very concrete, case-specific
way, should be compared with the theorem of 
{\sc Karlsson and Ledrappier}~\cite{KaLe}, which says that (under very general
conditions) the weak Liouville property holds if and only if the rate of
escape of Brownian motion is $0$.

When $\ab \ne 0$, we have the following.

\begin{cor}\label{cor:Poisson} \emph{(i)} If $\ab > 0$ then every bounded
harmonic function for $\Lap^{\Sol(\pp,\qq)}_{\ab}$ has the form
$\;(x,y,z) \mapsto h_2(y,-z)\,$, where $h_2$ is a bounded harmonic function
for $\Lap^{\Hb(\qq)}_{-\ab}$.\\[5pt]
\emph{(ii)} If $\ab < 0$ then every bounded
harmonic function for $\Lap^{\Sol(\pp,\qq)}_{\ab}$ has the form
$\;(x,y,z) \mapsto h_1(x,z)\,$, where $h_1$ is a bounded harmonic function
for $\Lap^{\Hb(\pp)}_{\ab}$.
\end{cor} 

\begin{proof} Let $h$ be a bounded harmonic function on $\Sol(\pp,\qq)$.
We may assume without loss of generality that it is non-negative. 
We decompose $h(x,y,z) = h_1(x,z) + h_2(y,-z)$ according to Corollary
\ref{cor:harmonic-sum}. Then both $h_1$ and $h_2$ are bounded harmonic.
When $\ab > 0$, Remark \ref{rmk:Liouville} tells us that $h_1$ must be constant, so
that we can ``incorporate'' it into $h_2$. Analogously, when $\ab < 0$, the function
$h_2$ must be constant.
\end{proof}

\begin{thm}\label{thm:main}
The minimal eigenfunctions in $\Hc^+\bigl(\Lap^{\Sol(\pp,\qq)}_{\ab},\la\bigr)$,
$\la \ge \la_{\min}$, are precisely the functions
$$
(x,y,z) \mapsto P_{\pp,\ab,\la}\bigl((x,z),\xi)
\AND (x,y,z) \mapsto P_{\qq,-\ab,\la}\bigl((y,-z),\eta)\,,\quad \xi,\eta \in \R\,,
$$
and in addition, when $\la = \la_{\min}$, the function
$$
(x,y,z) \mapsto e^{-\ab\, z}\,.
$$
\end{thm}

\begin{proof}
Combining Theorem \ref{thm:minimal} with Lemma \ref{lem:transfer},
we see that each minimal $\la$-eigenfunction on $\Sol$ must be of the
form
$$
\begin{aligned}
(x,y,z) &\mapsto P_{\pp,\ab,\la}\bigl((x,z),\xi)\,,\quad\text{where}
\quad \xi \in \vartheta \Hb(\pp)\,,\quad \text{or}\\
(x,y,z) &\mapsto P_{\qq,-\ab,\la}\bigl((y,-z),\eta)\,,\quad\text{where}\quad
\eta \in \vartheta \Hb(\qq)\,.
\end{aligned}
$$
We have to show that for $\xi \ne \varpi_{\pp}$ and for 
$\eta \ne \varpi_{\qq}\,$, the respective functions are indeed all minimal. 
Furthermore, we have to show
that for $\xi = \varpi_{\pp}$ and for $\eta = \varpi_{\qq}\,$, the two 
resulting functions are \emph{not} minimal $\la$-eigenfunctions on $\Sol$,
unless $\la = \la_{\min}\,$. In this last case both coincide and are equal
to $e^{-\ab\, z}$.  

\smallskip

So first we show minimality of 
$(x,y,z) \mapsto P_{\pp,\ab,\la}\bigl((x,z),\xi)$ with 
$\xi \in \vartheta^* \Hb(\pp)$.
Suppose that $P_{\pp,\ab,\la}\bigl((x,z),\xi) \ge h(x,y,z)$ for all
$\zf = (x,y,z)$, where $h \in \Hc^+\bigl(\Lap^{\Sol(\pp,\qq)}_{\ab},\la\bigr)$.

We decompose $h(x,y,z) = h_1(x,z) + h_2(y,-z)$ according to Corollary 
\ref{cor:harmonic-sum}. By minimality of 
$P_{\pp,\ab,\la}\bigl((\cdot,\cdot),\xi)$ 
in $\Hc^+\bigl(\Lap^{\Hb(\pp)}_{\ab},\la\bigr)$ (Lemma \ref{lem:transfer}),
we must have $h_1 = c \cdot P_{\pp,\ab,\la}\bigl((\cdot,\cdot),\xi)$, 
where $0 \le c \le 1$. If $c=1$, we are done.
If $c < 1$ then we get 
$$
P_{\pp,\ab,\la}\bigl((x,z),\xi\bigr) \ge \frac{1}{1-c} h_2(y,-z) 
= \int_{\vartheta \Hb(\qq)} P_{\qq,-\ab,\la}\bigl((x,-z),\eta) \, d\nu(\eta)
\quad \text{for all}\; (x,y,z) \in \Sol\,,
$$
where $\nu$ is a Borel measure on $\vartheta \Hb(\qq)$. Setting $y=z=0$,
we get 
$$
P_{\pp,\ab,\la}\bigl((x,0),\xi\bigr) \ge \nu\bigl(\vartheta \Hb(\qq)\bigr)
\quad \text{for all}\; x \in \R\,.
$$
If $x \to \infty$, then we see from the formula for $P_{\qq,-\ab,\la}$
of Lemma \ref{lem:transfer} that the left hand side in the last inequality
tends to $0$. Therefore $\nu\bigl(\vartheta \Hb(\qq)\bigr) = 0$, whence
$h_2 \equiv 0$, contradicting the assumption that $c < 1$.

\smallskip

The proof of minimality of $(x,y,z) \mapsto P_{\qq,-\ab,\la}\bigl((x,-z),\eta)$,
where $\eta \in \vartheta^* \Hb(\qq)$, follows as usual by exchanging the 
roles of the $x$- and $y$-variables.

\smallskip

Next, let $\xi = \varpi_{\pp}$ and $\la > \la_{min}$, 
so that we are considering the function
$$
(x,y,z) \mapsto P_{\pp,\ab,\la}\bigl((x,z),\varpi_{\pp}) =
e^{\alpha(\la,\ab)z}\,.
$$
If it were minimal in $\Hc^+\bigl(\Lap^{\Sol(\pp,\qq)}_{\ab},\la\bigr)$,
then by Lemma \ref{lem:lift-minimal}, also the function 
$(y,z) \mapsto e^{-\alpha(\la,\ab)z}$ would have to be minimal in  
$\Hc^+\bigl(\Lap^{\Hb(\qq)}_{-\ab},\la\bigr)$, which is not the case
by Lemma \ref{lem:transfer}.

Analogously, when $\la > \la_{min}$, 
the function
$$
(x,y,z) \mapsto P_{\qq,-\ab,\la}\bigl((y,-z),\varpi_{\qq})
e^{-\alpha(\la,-\ab)z}
$$
cannot be minimal in $\Hc^+\bigl(\Lap^{\Sol(\pp,\qq)}_{\ab},\la\bigr)$.

\smallskip 

Finally, consider the case $\la = \la_{min}$ and the function 
$(x,y,z) \mapsto e^{-\ab z}$
in $\Hc^+\bigl(\Lap^{\Sol(\pp,\qq)}_{\ab},\la\bigr)$.
We use a well-known trick, conjugating our operator with this exponential:
suppose that $e^{-\ab z} \ge h(x,y,z)$, where 
$h \in \Hc^+\bigl(\Lap^{\Sol(\pp,\qq)}_{\ab},\la\bigr)$.
Then a straightforward computation shows that
the function $\tilde h(x,y,z) = e^{\ab z}h(x,y,z)$ is in 
$\Hc^+\bigl(\Lap^{\Sol(\pp,\qq)}_{0},0\bigr)$, that is, it is bounded harmonic,
and the new rate of escape is $0$. By Corollary \ref{cor:Liouville},
$\tilde h$ is constant. This proves minimality of 
$(x,y,z) \mapsto e^{-\ab z}$
in $\Hc^+\bigl(\Lap^{\Sol(\pp,\qq)}_{\ab},\la\bigr)$.
\end{proof}

Our results tell us that the \emph{Poisson boundary} of Brownian motion with
drift on $\Sol$ is the ``outer'' boundary  
$$
\Bigl(\vartheta^* \Hb(\pp) \times \{\varpi_{\qq}\}\Bigr) \cup
\Bigl(\{\varpi_{\pp}\} \times \vartheta^* \Hb(\qq)\Bigr) \cup
\Bigl\{ (\varpi_{\pp}\,,\varpi_{\qq})\Bigr\}
$$
together with the limit distribution provided by Proposition \ref{pro:finlim}.
Indeed, for $\ab < 0$, it is just the first of these three pieces, because
the limit distributition is supported by that piece. For $\ab > 0$,
it is just the second piece, and for $\ab =0$, it is trivial, i.e., the singleton
of the third piece.
Here, we do not go into details regarding the construction of the Poisson boundary.
(In short, it is the largest probability space that gives rise to an integral
representation of all bounded harmonic functions and at the same time provides
a model for the limit behavior of the process at infinity.) The reader is referred
to the body of work of {\sc Kaimanovich}, e.g. \cite{Kai}. 

Regarding the Martin boundary (which is a metric space, while the Poisson boundary
is a measure space), our results underline the evidence that $\Mart(\la_{\min})$
is the boundary in the geometric compactification that we have described
in \S \ref{sect:potential}, while for $\la > \la_{\min}$ it should be bigger: one first
should consider the \emph{horocyclic compactification} of $\Hb(\pp)$, which 
can be built from the usual one as follows. Replace the boundary point $\varpi_\pp$
by the set $\{\varpi_{\pp}^{\zeta} : \zeta \in [-\infty\,,\,\infty]\}$, which
carries the topology of the extended real line. Furthermore,
modify the topology by saying that in the new compactification, 
$(x,z) \to \varpi_{\pp}^{\zeta}$ if $|x| \to \infty$ and $z \to \zeta$.
Then we expect that the Martin compactification of $\Sol(\pp,\qq)$ for $\la >
\la_{\min}$ is the closure of $\Sol$ in the direct product of the horocyclic
compactifications of the two hyperbolic planes. This evidence comes from the
strong analogy with the $\DL$-graphs (the horocyclic product of two 
homogeneous trees), see \cite{BrWo1}; the rigorous proof still has to be 
carried out.

\end{document}